\documentclass[11pt]{article}

\usepackage[utf8]{inputenc}
\usepackage{amsmath}
\usepackage[english]{babel}
\usepackage{amssymb}
\usepackage{amsthm}
\usepackage{mathtools}
\usepackage{esdiff}
\usepackage{microtype}
\usepackage{paralist}
\usepackage{skak}
\usepackage{bbm}
\usepackage[paper=a4paper,left=25mm,right=25mm,top=25mm,bottom=25mm]{geometry}
\usepackage{stmaryrd}
\usepackage{hyperref}
\usepackage{color} 
\usepackage{fancyhdr}
\usepackage{nomencl}
\usepackage{cite}
\usepackage{longtable}

\usepackage{thmtools}
\declaretheoremstyle[headfont=\normalfont\bfseries]{bfthmstyle}
\declaretheorem[numberwithin=section,style=bfthmstyle]{Theorem}
\declaretheorem[sharenumber=Theorem,style=bfthmstyle]{Lemma}
\declaretheorem[sharenumber=Theorem,style=bfthmstyle]{Proposition}
\declaretheorem[sharenumber=Theorem,style=bfthmstyle]{Corollary}

\declaretheoremstyle[headfont=\normalfont\bfseries,qed={$\diamond$}]{bfdefstyle}
\declaretheorem[sharenumber=Theorem,style=bfdefstyle]{Definition}
\declaretheorem[sharenumber=Theorem,style=bfdefstyle]{Definition-Proposition}
\declaretheorem[sharenumber=Theorem,style=bfdefstyle]{Definition-Lemma}

\newtheorem{definition-proposition}[Theorem]{Definition-Proposition} 
\newtheorem{definition-lemma}[Theorem]{Definition-Lemma} 

\declaretheoremstyle[headfont=\normalfont\bfseries,qed={$\diamond$}]{bfremstyle}

\newcommand{\rk}{\mathrm{rk}\,}
\newcommand{\im}{\mathrm{im}\,}

\newcommand{\norm}[1]{\left\Vert #1 \right\Vert}

\newcommand{\set}[1]{\left\lbrace #1 \right\rbrace}

\newcommand{\setdef}[2]{\left\{\ #1\ \left|\ \vphantom{#1} #2 \right.\right\}}

\renewcommand{\phi}{\varphi}

\declaretheoremstyle[
notefont=\normalfont, notebraces={}{},
headformat=\mathbb{N}UMBER~\mathbb{N}AME~\mathbb{N}OTE
]{nopar}

 \numberwithin{equation}{section}

\makenomenclature

\author{Achim Ilchmann, Jonas Kirchhoff, Manuel Schaller}
\title{Port-Hamiltonian descriptor systems are \\
relative generically controllable and stabilizable}
\date{\today}

\begin{document}
\setlength{\parindent}{0em}
\pagestyle{fancy}
\lhead{Achim Ilchmann, Jonas Kirchhoff, Manuel Schaller}
\rhead{Genericity}

\maketitle

\paragraph{Abstract} The present work is a successor of~\cite{IlchKirc21}
 on generic
controllability and of~\cite{IlchKirc22} on
relative generic controllability
of linear differential-algebraic equations.
We extend the result from general, unstructured differential-algebraic equations to differential-algebraic equations of port-Hamilto\-nian type.
We derive new results on relative genericity.
These findings are the basis for characterizing relative generic controllability
 of port-Hamilto\-nian systems in terms of dimensions.
 A similar result is proved for relative generic stabilizability.\vspace{-1mm}
 
\paragraph{Keywords} differential-algebraic equation\ $\cdot$\ port-Hamiltonian system\ $\cdot$\ controllability\ $\cdot$\ stabilizabilty\ $\cdot$\ genericity\ $\cdot$\ relative-genericity\vspace{-1mm}

\paragraph{Nomenclature}\ \\[3mm]
\begin{tabular}[ht]{lp{1pt}p{330pt}}
{$\mathbb{N}$, \ $\mathbb{N}^*$} 
&& {$:= \set{0,1,2,\ldots}$,  \ $:= \set{1,2,\ldots} = \mathbb{N}\setminus\set{0}$, resp.}\\
{$\mathbb{R}$,~$\mathbb{C}$} && {the field of the real, complex numbers, resp.}\\
{$\mathbb{K}$} && {either the field of the real or complex numbers}\\
{$\mathbb{N}_{\leq j}^*$} && {$:= \set{1,\ldots,j},~j\in\mathbb{N}^*$}\\
{$R^{k\times\ell}$} && {the vector space of all~$k\times\ell$ matrices with entries in a ring~$R$}\\
$\mathbf{Gl}(\mathbb{R}^k)$ && general linear group of all invertible~$k\times k$ matrices with entries in~$\mathbb{R}$\\
$M_{i,j}\ (M_{i,\cdot},M_{\cdot,j})$ && the entry~$(i,j)$ (the~$i$-th row, ~$j$-th column, resp.) of a matrix~$M\in R^{k\times \ell}$, $i\in\mathbb{N}^*_{\leq k},j\in\mathbb{N}^*_{\leq\ell}$\\
{$\rk_FM$} && {the rank of~$M\in R^{k\times\ell}$ over the field~$F$}\\
$\norm{\cdot}_{p, q}$ && an operator norm on~$\mathbb{R}^{p\times q}$\\
$\mathbb{B}(x^0,\varepsilon)$&& {$:= \setdef{x\in X}{ \|x-x^0\|<\varepsilon}$, the \textit{open ball} with center~$x^0\in X$ and radius~$\varepsilon\in~(0,\infty)$ for a normed real or complex vector space~$(X,\norm{\cdot})$.}\\
$AC(\mathbb{R},\mathbb{K}^n)$&& the set of all absolutely continuous functions~$f:\mathbb{R}\to\mathbb{K}^n$\\
$\mathcal{L}^1_{\mathrm{loc}}(\mathbb{R},\mathbb{K}^n)$&& the set of all locally integrable  functions~$f:\mathbb{R}\to\mathbb{K}^n$\\
{$f^{-1}(A)$} && {$:= \set{x\in X: f(x)\in A}$, the  preimage of the set~$A\subseteq Y$ under the function~$f: X\to Y$}\\
{$\mathbb{R}[x_1,\ldots,x_n]$} && {$:= \set{\sum_{k = 0}^\ell a_k x_1^{\nu_{k,1}}\cdots x_n^{\nu_{k,n}}\ \big\vert\, \ell\in\mathbb{N},a_k\in\mathbb{R},\nu_{k,j}\in\mathbb{N}}$, the ring of (real) polynomials in~$n$ indeterminants}\\
\end{tabular}

\vfill
\par\noindent\rule{5cm}{0.4pt}\\
\begin{footnotesize}
Corresponding author: Jonas Kirchhoff\\[1em]
Achim Ilchmann\ $\cdot$\ Jonas Kirchhoff\ $\cdot$\ Manuel Schaller\\
Institut für Mathematik, Technische Universität Ilmenau, Weimarer Stra\ss e 25, 98693 Ilmenau, Germany\\
E-mail: achim.ilchmann\ $\cdot$\ jonas.kirchhoff\ $\cdot$\ manuel.schaller@tu-ilmenau.de\\[1em]
Jonas Kirchhoff thanks the Technische Universität Ilmenau and the Freistaat Thüringen for their financial support as part of the Thüringer Graduiertenförderung.
\end{footnotesize}

\newpage

%
\section{Introduction}\label{Sec:Intr}

In the present work we characterize (relative) genericity of  controllability and
stabilizability  of linear finite-dimensional port-Hamiltonian descriptor systems described by linear differential-algebraic input-state-output systems of the 
following form~\cite{mehrmann2022control}
\begin{equation}\label{eq:pH_DAE}
\begin{aligned}
\tfrac{\mathrm{d}}{\mathrm{d}t}Ex & = (J-R)Qx + Bu,
\end{aligned}
\end{equation}
where~$B\in\mathbb{K}^{\ell\times m}$,~$J = -J^*\in\mathbb{K}^{\ell\times\ell}$,~$\mathbb{K}^{\ell\times\ell}\ni R = R^*\geq 0$ 
and~$E,Q\in\mathbb{K}^{\ell\times n}$ 
with~$E^* Q = Q^*E\geq 0$. 
So two subjects meet: genericity of controllability/stabilizability and port-Hamiltonian descriptor systems. We give a brief literature review on both subjects in the following.
\\

After Rudolf Kalman had introduced and characterized
in~1960~\cite{Kalm60b,Kalm61}
controllability of linear systems of the form
\begin{equation}\label{eq:ABC}
\begin{aligned}
\tfrac{\mathrm{d}}{\mathrm{d}t}x & = Ax + Bu,
\end{aligned}
\end{equation}
where~$(A,B)\in  \mathbb{K}^{n\times n}\times\mathbb{K}^{n\times m}$
for~$n,m\in\mathbb{N}^*$,
Lee and Markus~\cite{LeeMark67} proved in~1967 that the set of all
controllable systems~\eqref{eq:ABC} 
is open and dense with respect to the Euclidean topology. 
In~1974, Wonham~\cite{Wonh74} showed
that this set is not only open and dense but also \textit{conull} w.r.t.\ the Lebesgue measure, i.e. its complement is a Lebesgue null set. In 2019, Hinrichsen and Oeljeklaus~\cite{HinrOelj19} improved Wonham's result. They proved that in the case $m\geq 2$ the set of controllable systems~\eqref{eq:ABC} is generically convex, i.e. the set of pairs of tuples $((A_1,B_1),(A_2,B_2))$ so that all the systems~\eqref{eq:ABC} associated to the tuples $(A_1+\lambda (A_2-A_1),B_1+\lambda(B_2-B_1))$, $\lambda\in [0,1]$, are controllable is generic. 

Recently, in~2021, Wonham's results were generalized by Ilchmann and Kirchhoff~\cite[Thm.\,2.3]{IlchKirc21} to systems described by 
differential-algebraic equations of the form
\begin{align}\label{eq:the_DAE}
\tfrac{\mathrm{d}}{\mathrm{d}t}Ex = Ax + Bu,
\end{align}
where~$(E,A,B)\in \Sigma_{\ell,n,m} := \mathbb{K}^{\ell\times n}\times\mathbb{K}^{\ell\times n}\times\mathbb{K}^{\ell\times m}$
for~$\ell,n,m\in\mathbb{N}^*$.
It was shown that the DAE-systems~\eqref{eq:the_DAE} are generically\footnote{
A set~$S\subseteq\mathbb{R}^n$ is called \textit{generic} if, and only if,
 there exist a proper algebraic variety 
$\mathbb{V}\in\mathcal{V}_n^{\text{prop}}(\mathbb{R})$ so that~$S^c\subseteq\mathbb{V}$.}
\begin{equation}\label{eq:list-contr}
\begin{tabular}{lcc}
freely initializable & if, and only if, &~$\ell\leq n+m$;\\[0.5\normalbaselineskip]
impulse controllable & if, and only if, &~$\ell\leq n+m$;\\[0.5\normalbaselineskip]
behavioural controllable & if, and only if, &~$\ell\neq n+m$;\\[0.5\normalbaselineskip]
completely controllable & if, and only if, &~$\ell< n+m$;\\[0.5\normalbaselineskip]
strongly controllable & if, and only if, &~$\ell< n+m$;
\end{tabular}
\end{equation}
and if one of the conditions on the right-hand side is violated, then the associated controllability notion is generically violated; see Proposition~\ref{Def-Prop:Contr} for a definition of the five controllability concepts. In 2022, Hinrichsen and Oeljeklaus~\cite{HinrOelj22} proved that linear time-invariant ordinary delay-differential equations are generically controllable.

However, the concept of genericity of controllability for DAE-systems has the drawback
that the set~$\Sigma_{\ell,n,m} $ is too ``large'':
if $\ell =n$, then in each arbitrarily small neighbourhood of
$(E,A,B)\in \Sigma_{\ell,n,m} $ there is some invertible~$E'\in\mathbb{K}^{n\times n}$
such that~$(E',A,B)$ is an ordinary differential equation.
To resolve this drawback,
Kirchhoff~\cite{Kirc21pp} introduced the concept of
relative genericity\footnote{See Definition~\ref{def:rel_gen} in the present note.},
and  Ilchmann and Kirchhoff~\cite[Thm.\,3.2]{IlchKirc22} showed in~2022  that a similar
characterization to~\eqref{eq:list-contr} holds for relative genericity of controllability
for various reference sets such as 
$ \set{(E,A,B)\in\Sigma_{\ell,n,m}\,\big\vert\,\rk_{\mathbb{R}} E \leq r}$
for~$r\in\mathbb{N}$.
 
In 2021, Kirchhoff~\cite[Prop.\,III.1]{Kirc21pp} has shown that the set of all 
controllable linear port-Hamiltonian systems described by 
\begin{equation}\label{eq:pH_ODE}
\begin{aligned}
\tfrac{\mathrm{d}}{\mathrm{d}t}x & = JQx + Bu,
\end{aligned}
\end{equation}
where~$B\in\mathbb{K}^{\ell\times m}$,~$J = -J^*\in\mathbb{K}^{\ell\times\ell}$,
and~$Q^*=Q\in\mathbb{K}^{\ell\times n}$,
is relative generic in the set of all systems~\eqref{eq:pH_ODE}.
\\
Similar results of the afore mentioned hold for stabilizability.
\\

 The brief literature review about systems port-Hamiltonian systems 
 is as follows.
Port-Hamiltonian systems were introduced in the seminal work~\cite{Maschke93}.
They provide a flexible and highly structured framework to model complex phenomena by means of input-state-output dynamical systems. 
Application areas encompass many physical domains, such as, e.g., mechanics, electrical circuits, thermodynamics, and multi-energy systems, cf.\ \cite{Warsewa21,Hauschild20,louati2020network} for recent applications.  Due to their energy-based nature, a port-Hamiltonian (pH) structure inherently ensures passivity and stability and thus enables powerful analytic tools \cite{Jacob12}, numerical methods \cite{mehrmann2022control,kotyczka2019discrete} or passivity-based control techniques \cite{van2000l2,OrteScha02}, to analyze, simulate or control the corresponding systems. In view of optimal control, it was recently shown that pH structures can be used to overcome challenges in singular optimal control~\cite{FaulMascPhilSchaWort21,FaulMascPhilSchaWort22}, allowing results that are not available for general input-state-output systems. Another advantage of the pH~model class is its closedness under power-conserving interconnection: A power-conserving coupling of pH systems is again port-Hamiltonian. Last, the class of pH systems can be shown to be particularly robust w.r.t.\ structured perturbations \cite{mehl2021distance}.
\\
Port-Hamiltonian models arise in various formulations, such as geometric (by means of, e.g., Dirac structures \cite{DalsVdS99,GorrMascZwar05,cervera2007interconnection} and Lagrangian subspaces~\cite{MascVdS20}) or explicit state space models, both in finite- \cite{SchaJelt14} and infinite-dimensional \cite{Jacob12,Rashad20} spaces.
\\

This is, so far, the state of the art in relative genericity of controllability/stabiliza\-bility 
of  descriptor port-Hamilonian systems.
In the present paper we turn to linear port-Hamiltonian
differential-algebraic systems described by~\eqref{eq:pH_DAE}.
This class is a generalization of~\eqref{eq:pH_ODE} and a 
restriction of~\eqref{eq:the_DAE}.
 Due to the particular structure in~\eqref{eq:pH_DAE},
 it is not readily possible to apply the results of~\cite[Theorem 2.3]{IlchKirc21}
 and of~\cite[Thm.\,3.2]{IlchKirc22}. 
 However, utilizing Kirchhoff's\cite{Kirc21pp}
    slight adaptation of Wonham's original concept,
 we characterize~--  with striking resemblance to~\eqref{eq:list-contr}~--
 genericity and relative genericity of controllability/stabilizability of the system class~\eqref{eq:pH_DAE}. The latter is the main result of the paper.
\\

Our article is organized as follows. 
In section~\ref{Sec:Contr}, we define subclasses of port-Hamiltonian systems;
recall the definition and algebraic characterizations of the five   controllability 
 and  three stabilizability concepts;
recall the definitions of  (relative) genericity 
and prove some properties that will be intensively  exploited in the 
proofs of the upcoming results. 
Section~\ref{Sec:Interm_results} is an intermediate section;
we derive various lemmata and a proposition 
to show that subsets of systems~\eqref{eq:pH_DAE}
which satisfy certain algebraic constraints (the latter induced
by the characterizations of controllability/stabilizability)
are relative generic in some reference sets of port-Hamiltonian systems.
These results  lay the ground   for the proofs of the main results
in section~\ref{Sec:Main_results-contr}.
The latter are necessary and sufficient criteria on the dimensions~$\ell,n$ and~$m$ under which port-Hamiltonian systems are relative generically controllable.
Similar results for relative generic stabilizability are shown in section~\ref{Sec:Main_results-stab}.
Lastly, we will give a brief overview  of possible future research in 
section~\ref{sec:outlook}.

\section{Controllability, stabilizability, and relative genericity}\label{Sec:Contr}
In this section, we define a class of port-Hamiltonian DAE systems, recall controllability and stabilizability concepts for differential-algebraic equations,
 and collect results on (relative) genericity needed for later purposes.

{\bf \subsection{Port-Hamiltonian systems classes}}\label{subsec:classes}
\noindent
A \textit{port-Hamiltonian descriptor system} is a differential-algebraic equation 
of the form~\eqref{eq:pH_DAE}.
If~$R>0$, then the resulting DAE is called \textit{dissipative}; otherwise it is called \textit{semi-dissipative}, see~\cite[Definition 4.1.1(xxii) and (xxiii)]{Bern18} and~\cite{AchlArnoMehr22}. We denote the sets of all matrix quintuples associated to (semi-)dissipative port-Hamiltonian systems by
\begin{equation}\label{eq:Sigma-sdH}
\Sigma_{\ell,n,m}^{sdH} := \set{(E,J,R,Q,B)\,\big\vert\,J=-J^*,R = R^*\geq 0,E^*Q = Q^*E\geq 0}
\end{equation}
and
\begin{equation}\label{eq:Sigma-dH}
\Sigma_{\ell,n,m}^{dH} := \set{(E,J,R,Q,B)\,\big\vert\,J=-J^*,R = R^*> 0,E^*Q = Q^*E\geq 0}.
\end{equation}
An important subclass of port-Hamiltonian descriptor systems are \textit{conservative} systems, i.e.~$R = 0$, denoted by
\begin{equation}\label{eq:Sigma-H}
\Sigma_{\ell,n,m}^{H} := \set{(E,J,Q,B)\in\mathbb{K}^{\ell\times n}\times\mathbb{K}^{\ell\times \ell}\times\mathbb{K}^{\ell\times n}\times\mathbb{K}^{\ell\times m}\,\big\vert\,J=-J^*,E^*Q = Q^*E\geq 0}.
\end{equation}

{\bf \subsection{Controllability and stabilizability}}
%
To define controllability and stabilizability we need to define the
\emph{behaviour} of   the differential-algebraic equation~\eqref{eq:the_DAE},
where~$\ell,n,m\in\mathbb{N}^*$,  that is the set
\begin{align*}
\mathfrak{B}_{[E,A,B]} := \set{(x,u)\in\mathcal{L}^1_{\mathrm{loc}}(\mathbb{R},\mathbb{K}^n)\times\mathcal{L}^1_{\mathrm{loc}}(\mathbb{R},\mathbb{K}^m)\,\left\vert \begin{array}{l}
Ex\in AC(\mathbb{R},\mathbb{K}^n)~\text{\ and}\\
\tfrac{\mathrm{d}}{\mathrm{d}t} (Ex) = Ax + Bu~\text{a.e.}
\end{array} \right.}
\end{align*}
wherein ``a.e.'' stands for ``almost everywhere'' with respect to the Lebesgue measure.

Next we recall the  five definitions and characterizations of controllability
for linear dif\-fer\-ential-algebraic systems~\eqref{eq:the_DAE}.

\begin{Proposition}[{\cite[Def.~2.1 and Cor.~4.3]{BergReis13_loc}}]\label{Def-Prop:Contr} 
Consider the differential-algebraic equation~\eqref{eq:the_DAE} for some~$\ell,n,m\in\mathbb{N}^*$.
Write
\begin{align*}
S_{\textit{\text{controllable}}} := \set{(E,A,B)\in\Sigma_{\ell,n,m}\,\big\vert\,~\eqref{eq:the_DAE}~\textit{\text{controllable}}}
\end{align*}
for each of the following controllability concepts defined and algebraically characterized as follows:
\begin{align*}
\begin{array}{lrclcl}
(E,A,B)\in S_{\text{freely initializable}}
&:\iff&&
\forall\, x^0\in\mathbb{K}^n~\exists\, (x,u)\in\mathfrak{B}_{[E,A,B]}
: \ x(0) = x^0\\[2ex]
&\iff&& \rk[E,B] = \rk[E,A,B]\ ;\\[4ex]
(E,A,B)\in S_{\text{impulse controllable}}
&:\iff&&
\forall\, x^0\in\mathbb{K}^n~\exists\, (x,u)\in\mathfrak{B}_{[E,A,B]}
: \ Ex^0 = Ex(0)\\[2ex]
&\iff& &
\forall\, Z\in\mathbb{K}^{n\times n-\rk E}~\text{with}~\mathrm{im}_{\mathbb{K}} Z = \ker_{\mathbb{K}}E \\
&&& : \  \rk[E,A,B] = \rk[E,AZ,B]\  ;\\[4ex]
\end{array}
\end{align*}
\begin{align*}
\begin{array}{lrclcl}
(E,A,B)\in S_{\text{behavioural controllable}}
&:\iff&&
\forall\, (x_1,u_1),(x_2,u_2)\in\mathfrak{B}_{[E,A,B]}
\ \exists\, T>0~\exists\, (x,u)\in\mathfrak{B}_{[E,A,B]}\\[2ex]
&&&: \
(x,u)(t) = \begin{cases}(x_1,u_1)(t), & t<0\\(x_2,u_2)(t), & t>T\end{cases}\\[3ex]
&\iff&& 
\forall\, \lambda\in\mathbb{C}\ : \rk_{\mathbb{K}(x)}[xE-A,B]  
= \rk_{\mathbb{C}}[\lambda E-A,B]\  ;\\[4ex]
(E,A,B)\in S_{\text{completely controllable}}
&:\iff&&
\exists\, T>0~\forall\, x^0,x_T\in\mathbb{R}^n~\exists\, (x,u)\in\mathfrak{B}_{[E,A,B]}\\
&&& : \ x(0) = x^0\ \wedge \ x(T) = x_T \\[2ex]
&\iff&& 
\forall\,\lambda\in\mathbb{C}: \rk [E,A,B] = \rk [E,B] = \rk[\lambda E-A,B] \ ;\\[4ex]
(E,A,B)\in S_{\text{strongly  controllable}}
&:\iff&&
\exists\, T>0~\forall\, x^0,x_T\in\mathbb{R}^n
~\exists\, (x,u)\in\mathfrak{B}_{[E,A,B]}\\[2ex]
&&& :\ Ex(0) = Ex^0 \ \wedge\ Ex(T) = Ex_T\\[2ex]
&\iff&& 
\forall\, Z\in\mathbb{K}^{n\times (n-\rk E)}\
 \text{with}~\im_{\mathbb{K}} Z = \ker_{\mathbb{K}}E~\forall\,\lambda\in\mathbb{C}\\[2ex]
&&&  : \ \rk[E,A,B] = \rk[E,AZ,B]   = \rk [\lambda E-A,B] \ .\\[2ex]
\end{array}
\end{align*}
Furthermore,
\begin{align*}
S_{\text{completely~controllable}} & = S_{\text{freely initializable}}\cap S_{\text{behavioural controllable}}\quad\text{and}\\
S_{\text{strongly~controllable}} & = S_{\text{impulse~controllable}}\cap S_{\text{behavioural controllable}}.
\end{align*}
For port-Hamiltonian systems~\eqref{eq:pH_DAE}, we use, analogously to general differential-algebraic equations, the notation
\begin{equation}\label{eq:Hcontr}
S_{\text{\textit{controllable}}}^H := \set{(E,J,Q,B)\in\Sigma_{\ell,n,m}^H\,\big\vert\,(E,JQ,B)\in S_{\text{\textit{controllable}}}},
\end{equation}
and
\begin{equation}\label{eq:dHcontr}
S_{\text{\textit{controllable}}}^{dH} := \set{(E,J,R,Q,B)\in\Sigma_{\ell,n,m}^{dH}\,\big\vert\,(E,(J-R)Q,B)\in S_{\text{\textit{controllable}}}}
\end{equation}
where \textit{controllable} stands for either of the five aforementioned controllability concepts.
\end{Proposition}

\color{black}
\begin{proof}
Berger and Reis~\cite{BergReis13a}    derive a feedback  form and use this as a tool in conjunction with
`canonical' representatives of certain equivalence classes
to prove all characterizations 
of controllability in their survey.  
Note that   in their characterization of strongly controllability, the term~`$+\im_\mathbb{R} E$' is missing in the first respective line in~\cite[Cor.\,4.3]{BergReis13a}.
\end{proof}

Additionally to controllability, we will characterize genericity of stabilizability for port-Hamiltonian descriptor systems. To this end, we recall the definitions and algebraic characterizations. 

\begin{Proposition}[{\cite[Def.~2.1 and Cor.~4.3]{BergReis13_loc}}]
\label{Prop:Stab}
Consider the differential-algebraic equation~\eqref{eq:the_DAE} for some~$\ell,n,m\in\mathbb{N}^*$.
Writing, similar to the case of controllability,  
\begin{align*}
S_{\textit{\text{\textit{stabilizable}}}} := \set{(E,A,B)\in\Sigma_{\ell,n,m}\,\big\vert\,~\eqref{eq:the_DAE}~\textit{\text{\textit{stabilizable}}}},
\end{align*}
  the following definitions and characterizations  hold true:
\begin{align*}
\begin{array}{lrlcl}
(E,A,B)\in S_{\text{completely stabilizable}}
&:\iff& \forall\, x^0\in\mathbb{R}^n~\exists\, (x,u)\in\mathfrak{B}_{(E,A,B)}\\[2ex]
& & :\ x(0) = x^0 \ \wedge\ \lim_{t\to\infty}\mathrm{ess~sup}_{\tau\geq t}\norm{x(\tau)}= 0\\[2ex]
 & \iff&
 \forall\,\lambda\in\overline{\mathbb{C}}_{+}\ : 
  \rk_{\mathbb{R}}[E,A,B] = \rk_{\mathbb{R}}[E,B] =\rk_{\mathbb{C}}[\lambda E-A,B];\\[2ex]
(E,A,B)\in S_{\text{strongly stabilizable}}
&:\iff&
\forall\, x^0\in\mathbb{R}^n~\exists\, (x,u)\in\mathfrak{B}_{(E,A,B)}\\[2ex]
& & :\ Ex(0) = Ex^0  \ \wedge \   \lim_{t\to\infty}Ex(t) = 0\\[2ex]
&\iff& 
\forall\,\lambda\in\overline{\mathbb{C}}_+~\forall\,Z \in\mathbb{K}^{n\times (n-\rk E)}~\text{with}~\im Z = \ker E:\\[2ex]
&& \rk_{\mathbb{R}}[E,A,B] = \rk_{\mathbb{R}}[E,AZ,B]  = \rk_{\mathbb{C}}[\lambda E-A,B] ; \\[2ex]
(E,A,B)\in S_{\text{behavioural stabilizable}}
&:\iff&
\forall\, (x,u)\in\mathfrak{B}_{(E,A,B)}~\exists\, (x_1,u_1)
\in\mathfrak{B}_{(E,A,B)}\\[2ex]
&& :\ \left[\forall\, t<0:(x(t),u(t)) = (x_1(t),u_1(t))\right]\\[2ex]
& & \quad \wedge\quad
\lim_{t\to\infty}\mathrm{ess~sup}_{\tau\geq t}\norm{(x_1(\tau),u_1(\tau))} = 0 
\\[2ex]
&\iff& 
\forall\,\lambda\in\overline{\mathbb{C}}_+: \rk_{\mathbb{R}(x)}[xE-A,B] = \rk_{\mathbb{C}}[\lambda E-A,B].
\end{array}
\end{align*}
For port-Hamiltonian systems~\eqref{eq:pH_DAE}, we use  the notations
\begin{align*}
S_{\text{\textit{stabilizable}}}^H := \set{(E,J,Q,B)\in\Sigma_{\ell,n,m}^H\,\big\vert\,(E,JQ,B)\in S_{\text{\textit{stabilizable}}}}
\end{align*}
and
\begin{align*}
S_{\text{\textit{stabilizable}}}^{dH} := \set{(E,J,R,Q,B)\in\Sigma_{\ell,n,m}^{dH}\,\big\vert\,(E,(J-R)Q,B)\in S_{\text{\textit{stabilizable}}}}.
\end{align*}
\end{Proposition}
\begin{proof}
For the definitions see~\cite[Definition 2.1]{BergReis13_loc} and for the algebraic criteria see~\cite[Corollary 4.3]{BergReis13_loc}. Please note that although we adapted the definition of behavioural stabilizability, in view of the normal form under feedback equivalence~\cite[Theorem 3.3]{BergReis13_loc}, the algebraic characterization of Berger and Reis still holds, see also~\cite{IlchKirc22}.
\end{proof}

\color{black}

{\bf\subsection{Genericity and relative genericity}\label{subsec:Gen}}
\noindent
When it comes to genericity, there are various  approaches available in the literature. Purely topological, genericity is commonly defined as follows.

\begin{Definition}[{\cite[p.\,1]{ReicZasl14}}]\label{def:gen_topo}
Let~$(X,\mathcal{O})$ be a topological space. A set~$S\subseteq X$ is \textit{generic} if, and only if, it contains a \textit{residual set}, i.e. a family of open and dense sets~$O_n\in \mathcal{O}$, $n\in\mathbb{N}$, 
such that \ 
$\bigcap\limits_{n\in\mathbb{N}} O_n \subset S$.
\end{Definition}

If the topological space~$(X,\mathcal{O})$ is a Baire space~\cite[Definition 10.2, p.\,249]{Dugu70}, then all generic sets are dense. Additionally to the purely topological approach to genericity, there is also the measure theoretic point of view. In this context, Vovk defines a \textit{typical point}~\cite[p.\,274]{Vovk15pp}. This can be rephrased for genericity as follows.

\begin{Definition}
Let~$(X,\mathcal{A},\mu)$ be a complete (outer) measure space. A set~$S\subseteq X$ is called \textit{generic} if, and only if,~$S$ is \textit{$\mu$-conull}, i.e.~$S^c$ is measurable and~$\mu(S^c) = 0$.
\end{Definition}

It would be nice to have a statement  `If~$X$ admits a topology and~$\mu$ is a continuous Borel measure so that all open sets have nonzero measure, then genericity in the measure theoretic sense implies genericity in the topological sense.' However, a classical result says that the real line admits a residual nullset, see~\cite[Theorem 1.6, p.\,4]{Oxto71}. Ergo, there exists some set that is generic in the Euclidean topological sense, but its complement is generic in the Lebesgue measure theoretical sense. Since the Lebesgue measure is nontrivial, we see that our set is not generic in the Lebesgue measure theoretical sense. Conversely, by the Baire category theorem \cite[Theorem 10.1]{Dugu70}, 
 there is no \textit{meagre set} (i.e.~a set, whose complement is residual) that is also residual. Thus, the complement of our set cannot contain any residual set and is therefore not generic in the Euclidean topological sense. Thus, the two aforementioned genericity concepts are, in general, unrelated. To this end, we will utilize a concept of genericity that unites the topological and measure theoretic point of view.

Consider 
 the real finite-dimensional coordinate space~$\mathbb{R}^n$, which is intrinsically a Baire space -- equipped with the Euclidean topology -- and intrinsically a complete measure space -- equipped with the Lebesgue measure. In this case, 
Wonham introduced the classical concept of genericity   that unites the Euclidean topological and the Lebesgue meaure theoretic concepts of genericity using the Zariski topology as follows.

\begin{Definition}[{\cite[p.\,28]{Wonh74} and~\cite{Reid88}}]\label{def:gen}
Let~$n\in\mathbb{N}^*$.
A set~$\mathbb{V}\subseteq\mathbb{R}^n$ is called an \textit{algebraic set} if, and only if, 
\begin{align*}
\exists \ p_1,\ldots,p_k\in\mathbb{R}[x_1,\ldots,x_n] \ : \
\mathbb{V} = \bigcap_{i = 1}^k p_i^{-1}(\set{0}),
\end{align*}
i.e.\ algebraic sets are the locus of the common zeros of finitely many polynomials in~$n$ indeterminants.

A set~$S$ is called \textit{generic} in~$\mathbb{R}^n$ if, and only if, 
there is some strict algebraic set~$\mathbb{V}\subsetneq\mathbb{R}^n$ 
 so that~$S^c\subseteq \mathbb{V}$.
\end{Definition}

Let the polynomials~$p_i$ in Definition~\ref{def:gen} generate the ideal
\begin{align*}
I = \set{\left.\sum_{j = 1}^\kappa r_j p_{i_j}\,\right\vert\, 
\begin{array}{ll}
\kappa\in\mathbb{N}_0, \ r_1,\ldots,r_\kappa\in\mathbb{R}[x_1,\ldots,x_n],\\
i_1,\ldots,i_\kappa\in\mathbb{N}^*_{\leq k}
\end{array}
},
\end{align*}
i.e.\ an additive subgroup that is closed under multiplication with any polynomial. It is easy to verify that all members of~$I$ vanish at the algebraic set~$\mathbb{V}$. In later considerations, we make use of the correspondence of algebraic sets 
in~$\mathbb{R}^n$ and ideals in the ring of~$n$-variate polynomials. Since the latter is a commutative ring,  there is no need to distinguish between left and right ideals.

\begin{Lemma}[{\cite[pp.\,50/1]{Reid88}}]\label{lem:Zariski_und_ideale}
A set~$\mathbb{V}\subseteq\mathbb{R}^n$ is an algebraic set if, and only if, there is some ideal~$I\subseteq\mathbb{R}[x_1,\ldots,x_n]$ so that
\begin{align*}
\mathbb{V} = \set{x\in\mathbb{R}^n\,\big\vert\ \forall p\in I: p(x) = 0}.
\end{align*}
\end{Lemma}
Genericity in the sense of Definition~\ref{def:gen} can also be characterised by means of the \textit{Zariski topology}. This topology is given by its closed sets, that are precisely the algebraic sets, see~\cite[p.\,50]{Reid88}.

\begin{Lemma}\label{lem:gen_characterisation}
Let~$n\in\mathbb{N}^*$. As set~$S\subseteq\mathbb{R}^n$ is generic if, and only if, it contains a non-empty 
Zariski-open set~$O\subseteq\mathbb{R}^n$.
\end{Lemma}
\begin{proof}
A set~$O\subseteq\mathbb{R}^n$ is Zariski-open if, and only if,~$O^c$ is Zariski-closed. By definition of the Zariski topology, this is the case if, and only if,~$O^c$ is an algebraic set. Thus, the assertion holds true since~$O\subseteq S$ if, and only if,~$S^c\subseteq O^c$.
\end{proof}

\color{black}

The Zariski topology on~$\mathbb{R}^n$ is strictly coarser than the Euclidean topology and has the property that each non-empty Zariski-open set is not only Euclidean dense but also Lebesgue conull, see~\cite[p.\,240]{Fede69}. Further, there are residual conull sets that contain no non-empty Zariski open set (e.g. the natural numbers are an Euclidean closed Lebesgue nullset). Therefore Wonham's concept \textcolor{black}{from Definition~\ref{def:gen}} is strictly stronger than the Euclidean topological and Lebesgue measure theoretic approach to genericity. Additionally, from the tame nature of polynomials, the authors would like to argue that the concept of genericity in the sense of Definition~\ref{def:gen} is rather handy to work with as opposed to the natural choice of defining genericity as residual (or even open and dense) conull sets. Unfortunately, a drawback of Wonham's genericity as defined in Definition~\ref{def:gen} is that its naive extension to genericity with respect to some reference set~$V\subseteq\mathbb{R}^n$ fails in general. The relative Zariski topology on~$V$ has, in general, not the same favourable properties regarding the Euclidean topology (and the Lebesgue measure, if there is some reasonable way of defining such a measure) which allows to study genericity in the sense of Definition~\ref{def:gen} as a reasonable concept; the simplest example is a discrete set with at least two points. To overcome this issue, Kirchhoff used in~\cite{Kirc21pp} an adapted concept for genericity in some reference se, which was refined in~\cite{IlchKirc22} and is defined as follows.

\begin{Definition}[{\cite[Def.~I.2]{Kirc21pp} and~\cite[Def.~1.2]{IlchKirc22}}]\label{def:rel_gen}
Let~$n\in\mathbb{N}^*$ and~$S,V\subseteq\mathbb{R}^n$.~$S$ is \textit{relative generic} in~$V$ if, and only if, there is some algebraic set~$\mathbb{V}\subseteq\mathbb{R}^n$ so that
\begin{align*}
S^c\cap V\subseteq\mathbb{V}\cap V\quad\text{and}\quad\mathbb{V}^c\cap V~\text{is~Euclidean~dense~in}~V.
\end{align*}
\end{Definition}

Analogously to Wonham's original concept as given in Definition~\ref{def:gen}, relative genericity can be characterised in terms of the relative Euclidean and relative Zariski topologies as follows.

\begin{Lemma}[{\cite[Lemma 2.1]{IlchKirc22}}]\label{lem:rel_gen_char}
Let~$n\in\mathbb{N}^*$ and~$S,V\subseteq\mathbb{R}^n$.~$S$ is relative generic in~$V$ if, and only if,~$S\cap V$ contains some relative Zariski open, relative Euclidean dense set~$O\subseteq V$.
\end{Lemma}

In view of Lemma~\ref{lem:gen_characterisation}, relative genericity and genericity in the sense of Definition~\ref{def:gen} coincide for~$V = \mathbb{R}^n$. Moreover, when the reference set~$V$ admits enough structure (e.g. if~$V$ is an analytic submanifold with countable atlas), then the properties of the relative Zariski topology can be invoked to conclude that every relative generic set is conull w.r.t.\ a Lebesgue-type measure, see~\cite[Proposition 2.8]{IlchKirc22}. Thus, the authors prefer this concept of relative genericity over an adaptation of the purely topological concept of definition~\ref{def:gen_topo} w.r.t.\ the relative Euclidean topology on~$V$.

Ilchmann and Kirchhoff~\cite{IlchKirc22} have collected the results needed in the following considerations. We recall the most important ones here.

\begin{Proposition}[{see~\cite[Proposition 2.3]{IlchKirc22}}]\label{prop:properties_rel_gen}
Let~$S_1,S_2,V,V'\subseteq\mathbb{R}^n$,~$n\in\mathbb{N}^*$.\\[-4ex]
\begin{enumerate}[(a)]
\item If~$S_1$ is relative generic in~$V$, then~$S_1\cap V$ contains some open, dense set (in the relative Euclidean topology) and therefore~$S_1^c\cap V$ is nowhere dense in~$V$. The converse, however, is in general not true.
\item If~$S_1$ is relative generic in~$V$ and~$S_1\subseteq S_2$, then~$S_2$ is also relative generic in~$V$.
\item If~$S_1$ and~$S_2$ are relative generic in~$V$, then~$S_1\cap S_2$ and~$S_1\cup S_2$ are relative generic in~$V$.
\item 
If~$S_1$ is relative generic in~$V$ and~$S_3\subseteq\mathbb{R}^m$,~$m\in\mathbb{N}^*$, is relative generic in~$U\subseteq\mathbb{R}^m$, then~$S_1\times S_3$ is relative generic in~$V\times U$.
\item 
If~$V'\subseteq V$ and~$V'$ is relative generic in~$V$, 
then~$S_1$ is relative generic in~$V$ if, and only if, 
$S_1$ is relative generic in~$V'$. 
\item
If~$S_1$ is relative generic in~$V$, then~$S_1^c$ is not relative generic in~$V$.
\item If~$V$ is open, then~$S_1$ is relative generic in~$V$ if, and only if,~$S_1$ is generic.
\end{enumerate}
\end{Proposition}

The statement~(g) implies that relative genericity in open reference sets and genericity are equivalent. When considering relative open subsets of a given reference set~-- which is not necessarily~$\mathbb{R}^n$, as it was in~(g)~-- only one implication remains true.

\begin{Lemma}\label{lem:rel_open_reference_set}
Let~$V,\widetilde{V}\subseteq\mathbb{R}^n$ so that~$\widetilde{V}\subseteq V$ is relative Euclidean open. If~$S\subseteq\mathbb{R}^n$ is relative generic in~$V$, then~$S$ is relative generic in~$\widetilde{V}$. The converse  fails in general. 
\end{Lemma}
\begin{proof}
Let~$S$ be relative generic in~$V$. By Definition~\ref{def:rel_gen} there is some algebraic set~$\mathbb{V}\subseteq\mathbb{R}^n$ so that~$S^c\cap V\subseteq\mathbb{V}\cap V$ and~$\mathbb{V}^c\cap V$ is relative Euclidean dense in~$V$. Since~$\widetilde{V}$ is a subset of~$V$, we see that
\begin{align*}
S^c\cap\widetilde{V} = S^c\cap V\cap\widetilde{V} \subseteq \mathbb{V}\cap V\cap\widetilde{V} = \mathbb{V}\cap\widetilde{V}.
\end{align*}
It remains to show that~$\mathbb{V}^c\cap \widetilde{V}$ is relative Euclidean dense in~$\widetilde{V}$. Seeking a contradiction, assume that~$\mathbb{V}\cap \widetilde{V}$ contains an inner point~$x_0\in \mathbb{V}\cap \widetilde{V}$ w.r.t.~the relative Euclidean topology on~$\widetilde{V}$, i.e.
\begin{align*}
\exists\varepsilon>0: \quad \mathbb{B}(x_0,\varepsilon)\cap\widetilde{V}\subseteq\mathbb{V}\cap \widetilde{V}.
\end{align*}
Since~$\widetilde{V}$ is relative Euclidean open in~$V$, we have
\begin{align*}
\exists\varepsilon'\in (0,\varepsilon):\quad \mathbb{B}(x_0,\varepsilon')\cap V\subseteq\widetilde{V}.
\end{align*}
Since~$\varepsilon'<\varepsilon$, we have~$\mathbb{B}(x_0,\varepsilon')\cap V\subseteq \mathbb{B}(x_0,\varepsilon)$ and hence
\begin{align*}
\mathbb{B}(x_0,\varepsilon')\cap V\subseteq \mathbb{B}(x_0,\varepsilon)\cap\widetilde{V}\subseteq\mathbb{V}\cap \widetilde{V}\subseteq\mathbb{V}\cap V.
\end{align*}
Thus~$x_0$ is an inner point of~$\mathbb{V}\cap V$ w.r.t.~the relative Euclidean topology on~$V$~-- a contradiction to density of~$\mathbb{V}^c\cap V$ in~$V$ which is guaranteed by Definition~\ref{def:rel_gen}.

To see that the converse implication fails in general, consider the reference set~$V := \mathbb{B}(0,1)\cup\mathbb{R}\times\set{0}\subseteq\mathbb{R}^2$. The open ball~$\mathbb{B}(0,1)$ is a relative Euclidean open subset of~$V$. Further,~$\mathbb{B}(0,1)$ is open and hence by Proposition~\ref{prop:properties_rel_gen}\,(g) the set  $\mathbb{R}^2\setminus\mathbb{R}\times\set{0}$ is relative generic in~$\mathbb{B}(0,1)$. However,~$\mathbb{R}^2\setminus\mathbb{R}\times\set{0}$ is  not Euclidean dense in~$V$ and thus especially not relative generic in~$V$.
\end{proof}

An even simpler counter-example would be a disconnected reference set: Each connected component of any set~$V$ is relative generic in itself, but not relative generic in~$V$ provided that~$V$ possesses at least two connected components.

So far, we have considered genericity in the sense of Wonham only in the real coordinate space, a restriction that seems odd. However, one should be careful when extending Wonham's original definition (or the small adaptation towards arbitrary reference sets) to the complex case. Naively, the authors would try to use the complex Zariski topology (defined as in the real case but with complex polynomials) and get a reasonable concept. It should be kept in mind that the complex coordinate space~$\mathbb{C}^n$ has naturally a real vector space structure and is isomorphic to~$\mathbb{R}^{2n}$. When applying Wonham's genericity to the~$\mathbb{R}^{2n}$ representation of~$\mathbb{C}^n$,  we must be aware that there are~$2n$-dimensional real algebraic sets in~$\mathbb{C}^n$ that can only contain complex algebraic sets up to dimension~$n$. We must therefore distinguish between the complex genericity and the weaker concept of real genericity that is induced by any isomorphism between~$\mathbb{C}^n$ and~$\mathbb{R}^{2n}$. From the Euclidean topological and Lebesgue measure theoretic point of view, there seems to be no reason to prefer either of these concepts. For simplicity, however, we will use the real genericity for the complex coordinate space; a formal definition is the following.

\begin{Definition}\label{def:complex}
Let~$n\in\mathbb{N}^*$,~$S,V\in\mathbb{C}^n$ and~$\varphi:\mathbb{C}^n\to\mathbb{R}^{2n}$ a real vector space isomorphism. We call~$S$ \textit{relative generic} in~$V$ if, and only if,~$\varphi(S)$ is relative generic in~$\varphi(V)$. 
\end{Definition}

We stress that we will be very informal when identifying~$\mathbb{C}^n$ with~$\mathbb{R}^{2n}$ and equipping the former with its real Zariski topology. Since isomorphisms are bijective linear maps and especially polynomial vectors, Definition~\ref{def:complex} is independent of the explicit choice of~$\varphi$. Therefore we do not  state the used isomorphism explicitly. Since we have defined relative genericity in the complex coordinate space via its representation as real coordinate space, the results of Proposition~\ref{prop:properties_rel_gen} hold true for~$\mathbb{C}$ instead of~$\mathbb{R}$.

Before we proceed with our main results, we need an additional lemma on relative genericity.

\begin{Lemma}\label{lem:convex_reference_sets}
Let~$n\in\mathbb{N}^*$ and~$V\subseteq\mathbb{R}^n$ be convex and non-empty. Then~$S\subseteq \mathbb{R}^n$ is relative generic in~$V$ if, and only if, there is some Zariski open set~$O$ so that~$O\cap V\neq\emptyset$ and~$O\cap V\subseteq S\cap V$. 
\end{Lemma}
\begin{proof}
``$\implies$'': \ Let~$S$ be relative generic in~$V$. In view of Lemma~\ref{lem:rel_gen_char}, there is some relative Zariski-open, relative Euclidean dense set~$\widetilde{O}\subseteq V$ so that~$\widetilde{O}\subseteq S\cap V$. By definition of the relative Zariski-topology, there is some Zariski open set~$O\subseteq\mathbb{R}^n$ so that~$\widetilde{O} = O\cap V$. Hence, the inclusion~$O\cap V \subseteq S\cap V$ holds true. Since~$V$ is non-empty and~$\widetilde{O}$ is an Euclidean dense subset of~$V$,~$\widetilde{O}$ is non-empty. This yields~$O\neq\emptyset$.

\noindent
``$\impliedby$'':  \ 
Let~$O\subseteq\mathbb{R}$ be a Zariski-open set so that
\begin{align*}
    O\cap V\neq\emptyset\quad\text{and}\quad O\cap V\subseteq S\cap V.
\end{align*}
If~$V$ contains only one point, then we conclude from
\begin{align*}
\emptyset\neq O\cap V\subseteq S\cap V\subseteq V
\end{align*}
that~$S\cap V$ is a non-empty subset of~$V$ and thus~$S\cap V = V$. Therefore,~$S$ is relative generic in~$V$. In the following, we consider the case that~$V$ contains at least two (and hence, due to convexity, uncountably many) points. We split the proof into steps.

\noindent
\textsc{Step 1:} \quad We show that we can, without loss of generality,
assume that~$V$ has non-empty interior and spans~$\mathbb{R}^n$. Since relative genericity is by Definition~\ref{def:rel_gen} invariant under simultaneous translations of~$S$ and~$V$, we may  assume , without loss of generality, that~$0\in V$. Moreover, in view of Lemma~\ref{lem:rel_gen_char} it can be readily seen that~$S$ is relative generic in~$V$ if, and only if,~$S\cap V$ is relative generic in~$V$. Thus, we may   assume  without loss of generality that~$S\subseteq V$. In~\cite[Lemma III.2]{Kirc21pp}, it was shown that when considering relative genericity w.r.t.~a reference set that is contained in a linear subspace~$V\subseteq\mathbb{R}^n$, it can,  without loss of generality, be assumed that~$V = \mathbb{R}^n$. Hence, we can   assume  without loss of generality that~$\mathbb{R}^n$ and~$\mathrm{span}\,V$, the linear span of~$V$, coincide. In this case, it is well-known that~$V$ contains a basis~$(b_1,\ldots,b_n)\in V^n$ of~$\mathbb{R}^n$. By convexity of~$V$, we conclude that the convex hull of the set~$\set{0,b_1,\ldots,b_n}$, that is
\begin{align*}
C := \mathrm{conv}(\set{0,b_1,\ldots,b_n}) = \set{\sum_{i = 1}^n\lambda_i b_i\,\left\vert\,\lambda_1,\ldots,\lambda_n\in\mathbb{R}_{\geq 0}, \sum_{i = 1}^n\lambda_i\leq 1\right.}
\end{align*}
is contained in~$V$. Since the boundary of an~$n$-simplex is the union of its lower-dimensional faces,~$v_0 = \sum_{i = 1}^n\frac{1}{n+1}b_i$ 
is an inner point of~$C$ and hence the interior of~$V$,~$\mathrm{int}\,V$, is indeed non-empty.
\\

\noindent
\textsc{Step 2:}\quad We show
\begin{align}\label{eq:interior_convex}
\forall v\in V~\forall\varepsilon>0: \mathbb{B}(v,\varepsilon)\cap \mathrm{int}\,V \neq\emptyset.
\end{align}
Let~$v\in V$. Since~$v_0$ from Step\,1 is an inner point of~$V$, there is some~$\varepsilon_0>0$ so that~$\mathbb{B}(v_0,\varepsilon_0)\subseteq V$. Choose~$v_1,\ldots,v_n\in\mathbb{B}(v_0,\varepsilon_0)$ so that~$(v,v_1,\ldots,v_n)$ is affinely independent, i.e.~$(v_1-v,\ldots,v_n-v)$ is a basis of~$\mathbb{R}^n$.
Invoking convexity of~$V$, we see that
\begin{align*}
\widetilde{C} := \mathrm{conv}(\set{v,v_1,\ldots,v_n}) = \set{\left.\lambda v + \sum_{i = 1}^n\lambda_iv_i\,\right\vert\,\lambda,\lambda_1,\ldots,\lambda_n\in\mathbb{R}_{\geq 0},\lambda+\sum_{i = 1}^n\lambda_i = 1}\subseteq V.
\end{align*}
Further, we conclude from our choice of~$v_1,\ldots,v_n$ that~$\widetilde{C}$ is an~$n$-dimensional simplex in an~$n$-dimensional space and thus
\begin{align*}
\forall \lambda,\lambda_1,\ldots,\lambda_n\in\mathbb{R}_{>0}~\text{with}~\lambda+\sum_{i = 1}^n\lambda_i = 1: \lambda v + \sum_{i = 1}^n\lambda_iv_i\in\mathrm{int}\,\widetilde{C}\subseteq\mathrm{int}\,V.
\end{align*}
Hence the property~\eqref{eq:interior_convex} holds indeed true.

\noindent\textsc{Step 3:}\quad 
We prove that~$S\subseteq V$ is relative generic in~$V$. 
In view of Lemma~\ref{lem:rel_gen_char}, we may equivalently show
 that~$S$ contains some relative Zariski-open relative Euclidean dense set~$\widehat{O}\subseteq V$. By assumption, there is some Zariski-open set~$O\subseteq\mathbb{R}^n$ so that~$O\cap V\subseteq S = S\cap V$ and~$O\cap V\neq\emptyset$. Thus,~$O\cap V$ is, by definition of the relative topology, a relative Zariski open subset of~$S$. Hence, it remains to prove that~$O\cap V$ is Euclidean dense in~$V$. Seeking a contradiction, assume that~$O^c\cap V$ contains a relative \textit{inner point}, i.e.
\begin{align*}
\exists  \ x\in O^c\cap V \ \exists \  \varepsilon>0 \ : \ 
\mathbb{B}(x,\varepsilon)\cap V\subseteq O^c\cap V.
\end{align*}
In view of~\eqref{eq:interior_convex}, there is some~$x'\in\mathbb{B}\left(x,\frac{\varepsilon}{2}\right)\cap \mathrm{int}\,V$. By the triangle inequality, we 
conclude
\begin{align*}
\mathbb{B}\left(x',\tfrac{\varepsilon}{2}\right)\cap V
 \ \subseteq \ 
 \mathbb{B}(x,\varepsilon)\cap V
 \ \subseteq \ 
  O^c\cap V.
\end{align*}
Since~$x'$ is an inner point of~$V$, there is 
some~$\varepsilon' \in    (0, \frac{\varepsilon}{2})$ so that 
\begin{align*}
\mathbb{B}(x',\varepsilon') \ \subseteq \ 
\mathbb{B}\left(x',\tfrac{\varepsilon}{2}\right)\cap V 
\ \subseteq \ 
 O^c\cap V.
\end{align*}
Thus,  $x'$ is an inner point of~$O^c\cap V$ and therefore an inner point of~$O^c$. Since~$O$ is Zariski-open,~$O^c$ is an algebraic set and since~$O$ is non-empty,~$O^c$ is, in view of~\cite[p.\,240]{Fede69},  a Lebesgue nullset. Since every nontrivial open ball has nonzero Lebesgue measure,~$O^c$ must have empty
 interior~-- a contradiction to the existence of a relative inner point of~$O^c\cap V$. This shows that~$S$ is indeed relative generic in~$V$.

\color{black}
\end{proof}

\section{Relative generic sets}\label{Sec:Interm_results}

The present section is an intermediate section.
We   prove those results on relative genericity which are needed for proving the main result
in section~\ref{Sec:Main_results-contr}. 
The main result under these lemmata is Proposition~\ref{lem:main_result_1};
it shown that various subsets of differential-algebraic systems~\eqref{eq:pH_DAE}
satisfiying an algebraic constraint associated to controllability/stabilizability
are (relative) generic sets  with respect to (semi-)dissipative or conservative port-Hamiltonian des\-criptor systems.

The most important tool in our analysis is the well-known concept of a minor.

\begin{Definition}\label{def:minor}
Let~$d,\ell,n\in\mathbb{N}^*$ so that~$d\leq\min\set{\ell,n}$. Let
\begin{align*}
\sigma:\mathbb{N}^*_{\leq d}\to\mathbb{N}^*_{\leq \ell}
\end{align*}
and
\begin{align*}
\pi:\mathbb{N}^*_{\leq d}\to\mathbb{N}^*_{\leq n}
\end{align*}
be strictly increasing, where we use the abbreviation~$\mathbb{N}^*_{\leq k} := \set{1,\ldots,k}$ for all~$k\in\mathbb{N}^*$. We consider either~$R = \mathbb{K}$ or~$R = \mathbb{K}[x]$. The function
\begin{align*}
M_{\sigma,\pi}:R^{\ell\times n}\to R,\qquad A\mapsto \det[A_{\sigma(i),\pi(j)}]_{i,j\in\mathbb{N}_{\leq d}^*}
\end{align*}
is called \textit{minor of order}~$d$ w.r.t.~$R^{\ell\times n}$. The minor of order~$0$ w.r.t.~$R^{\ell\times n}$ is the constant function~$A\mapsto 1$.
\end{Definition}

Recall the correspondence between lower bounds of the rank of matrices and minors.

\begin{Lemma}[{see~\cite[Section 3.3.6]{Fisc05b}}]\label{lem:important_property}
Let~$d\in\mathbb{N}_0$ and~$\ell,n\in\mathbb{N}^*$ with~$d\leq\min\set{\ell,n}$ and either~$R = \mathbb{K}$ or~$R = \mathbb{K}[x]$. A matrix~$A\in R^{\ell\times n}$ has rank at least~$d$ if, and only if, there is some minor of order~$d$ w.r.t.~$R^{\ell\times n}$ that does not vanish at~$A$.
\end{Lemma}

Motivated by the algebraic characterizations of controllability and stabilizability in Proposition~\ref{Def-Prop:Contr} and~\ref{Prop:Stab}, we make use of Lemma~\ref{lem:important_property} to construct algebraic sets defined by the rank of certain polynomial matrices. This is possible since determinants and thus minors are in particular polynomials. 

\begin{Lemma}\label{lem:minors_are_polynomials}
Let~$d,g\in\mathbb{N}_0$ and~$\ell,n\in\mathbb{N}^*$ with~$d\leq\min\set{\ell,n}$. Let~$\widetilde{M}$ be a minor of order~$d$ w.r.t.~$\mathbb{K}[x]^{\ell\times n}$ and consider the induced function
\begin{align*}
M:\left(\mathbb{K}^{\ell\times n}\right)^{g+1}\to\mathbb{K}[x],\qquad (P_0,\ldots,P_g)\mapsto\widetilde{M}\left(\sum_{i = 0}^gP_i x^i\right).
\end{align*}
Then there are multivariate polynomials~$M_0,\ldots,M_{dg}\in\mathbb{K}[x_1,\ldots,x_{\ell n(g+1)}]$ 
so that
\begin{align}\label{eq:formula_minors}
\forall (P_0,\ldots,P_g)\in\left(\mathbb{K}^{\ell\times n}\right)^{g+1}: M(P_0,\ldots,P_g) = \sum_{i = 0}^{dg} M_i(P_0,\ldots,P_g)x^i.
\end{align}
In particular, 
\begin{align}\label{eq:particular_minors}
M_0(P_0,\ldots,P_g) = \widetilde{M}(P_0)\quad\text{and}\quad M_{gd}(P_0,\ldots,P_g) = \widetilde{M}(P_g).
\end{align}
\end{Lemma}
\begin{proof}
By definition, minors are determinants of certain submatrices. Therefore, it suffices to consider the case~$\ell = n = d$ and~$\widetilde{M} = \det(\cdot)$. In that case, the Leibniz formula for the determinant yields
\begin{align*}
M(P_0,\ldots,P_g) = \widetilde{M}\left(\sum_{i = 0}^gP_i x^i\right) = \sum_{\sigma\in S_d}\mathrm{sgn}(\sigma)\prod_{j = 1}^d\sum_{i = 0}^g(P_i)_{j,\sigma(j)} x^i,
\end{align*}
where~$S_d$ denotes the set of all permutations of all~$d$-tuples.
From this,~\eqref{eq:formula_minors} and~\eqref{eq:particular_minors} 
follow by expanding the products and sorting by the degree of the resulting monomials. The fact that the functions~$M_i$ are polynomials in the entries of~$P_0,\ldots,P_g$, follows likewise from the Leibniz formula.
\end{proof}

The properties of minors in Lemma~\ref{lem:important_property} and~\ref{lem:minors_are_polynomials} will help   to understand Zariski-open sets within the sets~$S^H_{\textit{controllable}}$ and~$S^{sdH}_{\textit{controllable}}$ that are given by the rank of particular block matrices. As a last algebraic ingredient, we recall the important concept of the Sylvester resultant.

\begin{Lemma}[{\cite[Thm.~3.3.1, p.\,61]{Fuhr12}}]\label{Def:Resultant}
The \emph{resultant} of two polynomials~$p,q\in\mathbb{K}[x]\setminus\set{0_{\mathbb{K}[s]}}$ with~$\deg p = n\geq 0$ and~$\deg q = m\geq 0$ and coefficients~$p_1,\ldots,p_n,q_1,\ldots,q_m\in\mathbb{K}$ is defined as
\[
\mathbb{R}es(p,q) = \det\underbrace{\left[\begin{array}{ccccc|ccccc}
p_0 & & & & & q_0 & & &\\
p_1 & p_0 & & & & q_1 & \cdot &\\
\cdot & \cdot & & & & \cdot & \cdot & \cdot & \\
\cdot & \cdot & \cdot & & & \cdot & \cdot & \cdot & q_0 \\
p_n & p_{n-1} & \cdot & \cdot & & q_n & \cdot & \cdot & q_1 \\
 & p_n & \cdot & \cdot & & \cdot & \cdot & \cdot & \cdot \\
 & & \cdot & \cdot & p_0 & q_{m-1} & \cdot & \cdot & \cdot\\
& & \cdot & \cdot & \cdot & q_{m} & \cdot & \cdot & \cdot & \\
& & \cdot & \cdot & \cdot & & \cdot & \cdot & \cdot \\
& & & \cdot & \cdot & & & \cdot & \cdot\\
& & & & p_n & & & &  q_m
\end{array}\right]}_{\in\mathbb{K}^{(n+m)\times (m+n)}}.
\]
The matrix above contains~$m$ columns with the coefficients of~$p$ and~$n$ columns with the coefficients of~$q$, so that it is in~$\mathbb{K}^{(n+m)\times (m+n)}$. All other entries are zero. Please note that the diagram illustrates the case~$n<m$.

The resultant of~$p$ and~$q$ vanishes if, and only if,~$p$ and~$q$ are not coprime~$\big($i.e.~there is some common zero~$z\in\mathbb{C}$ such that~$p(z) = q(z) = 0\big)$.
\end{Lemma}

Since relative genericity strongly depends on the reference set, we seek to simplify (in the sense of Proposition~\ref{prop:properties_rel_gen}\,(e)) the reference sets~$\Sigma_{\ell,n,m}^H$ and~$\Sigma_{\ell,n,m}^{sdH}$ as much as possible. As a first step towards a simplification, we prove that semi-definite matrices are generically definite.

\begin{Lemma}\label{lem:definite_is_generic}
Let~$\ell\in\mathbb{N}^*$. The set~$S = \set{M\in\mathbb{K}^{\ell\times \ell}\,\big\vert\,M^* = M > 0}$ of Hermitian (symmetric) positive definite matrices is relative generic in~$V = \set{M\in\mathbb{K}^{\ell\times \ell}\,\big\vert\,M^* = M \geq 0}$, the set of Hermitian (symmetric) positive semidefinite matrices. The same holds true if we consider negative instead of positive (semi-)definite matrices.
\end{Lemma}
\begin{proof}
Recall that each symmetric positive semi-definite matrix is not positive definite if, and only if, it has non-empty kernel (see e.g.~\cite[Theorem 7.2.6]{HornJohn12}). Since the determinant is a polynomial,~$S = V\cap S = V\setminus\set{M\in\mathbb{K}^{\ell\times \ell}\,\big\vert\,\det M = 0}$ is relative Zariski-open. In view of Lemma~\ref{lem:rel_gen_char}, it remains to prove that~$S$ is a Euclidean dense subset of~$V$. Let~$M\in V$ and~$\varepsilon>0$. We construct a symmetric positive definite matrix~$\tilde{M}\in S$ so that~$\norm{M-\tilde{M}}<\varepsilon$, where~$\norm{\cdot}$ is a norm on~$\mathbb{K}^{\ell\times \ell}$ so that~$\norm{I_\ell} = 1$. It is evident that~$M+\frac{\varepsilon}{2} I_\ell$ is positive definite and
\begin{align*}
\norm{M-M-\tfrac{\varepsilon}{2} I_\ell} = \tfrac{\varepsilon}{2}<\varepsilon.
\end{align*}
This shows the assertion.
\end{proof}

When studying the reference sets~$\Sigma_{\ell,n,m}^H$ and~$\Sigma_{\ell,n,m}^{sdH}$, one of the properties of the matrix tuples~$(E,J,Q,B)$ and~$(E,J,R,Q,B)$, resp., is the symmetry condition~$E^*Q = Q^*E$. Given matrices~$E,Q\in\mathbb{K}^{\ell\times n}$ with~$E^*Q = Q^*E$, we show how perturbations of~$E$ and~$Q$ that preserve this symmetry property look like.

\begin{Lemma}\label{lem:How_do_feasible_perturbations_look_like}
Let~$\ell,n\in \mathbb{N}^*$ and~$E,Q\in \mathbb{K}^{\ell\times n}$ with~$E^*Q = Q^*E$. Then there exist unitary (orthogonal) matrices~$P\in \mathbb{K}^{\ell\times \ell}$,~$T\in \mathbb{K}^{n\times n}$ and
\begin{align}\label{eq:proof1}
    \Sigma_k = \begin{bmatrix}
\sigma_1\\
& \ddots\\
& & \sigma_k
\end{bmatrix} \in\mathbf{Gl}(\mathbb{R}^k),\quad \text{where } k = \rk E
\end{align}
such that, for some~$\widetilde{Q}\in \mathbb{K}^{k\times k}$,~$R_1\in \mathbb{K}^{(\ell-k)\times k}$,~$R_2\in \mathbb{K}^{(\ell-k)\times(n-k)}$,
\begin{align}\label{eq:proof2}
    PET = \begin{bmatrix}
    \Sigma_k&0_{k\times (n-k)}\\
    0_{(\ell-k)\times k}&0_{(\ell-k)\times (n-k)}
    \end{bmatrix},\qquad PQT= \begin{bmatrix}
    \widetilde{Q}&0_{k\times(n-k)}\\
    R_1&R_2
    \end{bmatrix}
\end{align}
and 
\begin{align}\label{eq:proof3}
    TE^*QT = \begin{bmatrix}
    \Sigma_K \widetilde{Q}&0_{k\times(n-k)}\\
    0_{(\ell-k)\times k}&0_{(\ell-k)\times (n-k)}
    \end{bmatrix}
\end{align}
and
\begin{align}\label{eq:proof4}
    \Sigma_k \widetilde{Q} = \widetilde{Q}^* \Sigma_k, \quad \rk \widetilde{Q} = \rk E^* Q
\end{align}
and 
\begin{align}\label{eq:proof5}
    E^*Q=Q^*E \geq 0 \quad \text{if, and only if, } \quad \Sigma_K\widetilde{Q} = \widetilde{Q}^*\Sigma_k \geq 0.
\end{align}
Furthermore, 
\begin{align*}
    V_{\Sigma_k} := \left\{ \left.\begin{bmatrix}
    M&0_{k\times (n-k)}\\
    H_1&H_2
    \end{bmatrix} \in \mathbb{K}^{\ell\times n}\,\right\vert\, \Sigma_k M = M^*\Sigma_k\right\}
\end{align*}
is a real vector space.
\end{Lemma}
\begin{proof}
By the singular value theorem~\cite[Theorem 2.6.3]{HornJohn12}, there are orthogonal (unitary) matrices~$P\in \mathbb{K}^{\ell\times \ell}$ and~$T\in \mathbb{K}^{n\times n}$ such that~\eqref{eq:proof2} holds with~$\Sigma_k$ as in~\eqref{eq:proof1}. Since $P$ is unitary (orthogonal) and~$T$ invertible,~$E^*Q = Q^*E$ is equivalent to 
\begin{align}\label{eq:proof7}
    (PET)^*(PQT) = (PQT)^*(PET).
\end{align}
Inserting the first equation of~\eqref{eq:proof2} into \eqref{eq:proof7} yields the second equation in~\eqref{eq:proof2}.\\
This gives 
\begin{align}\label{eq:proof8}
    T^*E^*QT = (PET)^*(PQT) \stackrel{\eqref{eq:proof2}}{=} \begin{bmatrix}
    \Sigma_k \widetilde{Q}&0_{k\times(n-k)}\\
    0_{(\ell-k)\times k}&0_{(\ell-k)\times(n-k)}
    \end{bmatrix}
\end{align}
which shows~\eqref{eq:proof3}. As~$\Sigma_k \in\mathbf{Gl}(\mathbb{R}^k)$, the second equation in~\eqref{eq:proof4} follows from~\eqref{eq:proof8}. A repeated application of~$E^*Q=Q^*E$ yields
\begin{align*}
    \begin{bmatrix}
    \Sigma_k \widetilde{Q}&0_{k\times (n-k)}\\
    0_{(\ell-k)\times k}&0_{(\ell-k)\times(n-k)}
    \end{bmatrix}
&\stackrel{\mathclap{\eqref{eq:proof8}}}{=} T^*E^*QT = T^*Q^*ET = (PQT)^*(PET)\\
&=   \begin{bmatrix}
    \widetilde{Q}^*\Sigma_k &0_{k\times(n-k)}\\
    0_{(\ell-k)\times k}&0_{(\ell-k)\times(n-k)}
    \end{bmatrix}
\end{align*}
which shows the first equation in~\eqref{eq:proof4}. The equivalence in~\eqref{eq:proof5} is a consequence of~\eqref{eq:proof8}. Finally, it is easy to see that~$V_{\Sigma_k}$ is a real vector space.
\end{proof}

Next, we show that the matrices~$E$ and~$Q$, from which the generalized energy function~$x\mapsto x^*E^*Qx$ that is associated to a port-Hamiltonian descriptor system is composed of, have generically full rank.

\begin{Lemma}\label{lem:Hamiltonian_pairs_full_rank_pointwise}
Let~$\ell,n\in\mathbb{N}^*$. Then each of the  sets 
\begin{enumerate}[(i)]
\item~$S_{(i)}\, = \set{(E,Q)\in\mathbb{K}^{\ell\times n}\times\mathbb{K}^{\ell\times n}\,\big\vert\,\rk Q = \min\set{\ell,n}}$ 
\item~$S_{(ii)} = \set{(E,Q)\in\mathbb{K}^{\ell\times n}\times\mathbb{K}^{\ell\times n}\,\big\vert\,\rk E = \min\set{\ell,n}}$
\end{enumerate}
is relative generic in the reference set
\[
V \ = \ \set{(E,Q)\in\mathbb{K}^{\ell\times n}\times\mathbb{K}^{\ell\times n}\,\big\vert\,E^*Q = Q^*E}.
\]
\end{Lemma}
\begin{proof}
Since the statements of (i) and~(ii) are symmetric, it suffices in view of Proposition~\ref{prop:properties_rel_gen}\,(b) to show that~$S_{(i)}$ is relative generic in~$V$. Due to Lemmata~\ref{lem:important_property}~and~\ref{lem:minors_are_polynomials}, the set 
$$
\set{(E,Q)\in\mathbb{K}^{\ell\times n}\times\mathbb{K}^{\ell\times n}\,\big\vert\,\rk Q \geq \min\set{\ell,n}}
$$
is Zariski-open. As~$\rk Q \leq \min\set{\ell,n}$ for all~$Q \in \mathbb{K}^{\ell \times n}$,~$S_{(i)}$ is Zariski-open, as well. By Lemma~\ref{lem:rel_gen_char}, it suffices to show that~$S_{(i)}\cap V$ is Euclidean dense in~$V$. Let~$(E,Q)\in V$. Lemma~\ref{lem:How_do_feasible_perturbations_look_like} yields that there are unitary (orthogonal) matrices~$P\in\mathbb{K}^{\ell\times\ell}$ and~$T\in\mathbb{K}^{n\times n}$ and some diagonal matrix~$\Sigma\in\mathbf{Gl}(\mathbb{R}^k)$ with~$k = \rk E$ so that
\begin{align*}
PET = \begin{bmatrix}
\Sigma & 0_{k\times (n-k)}\\
0_{(\ell-k)\times k} & 0_{(\ell-k)\times (n-k)}
\end{bmatrix}
\end{align*}
and~$SQ'T\in V_{\Sigma}$ as introduced in Lemma~\ref{lem:How_do_feasible_perturbations_look_like}. We show that the set
\begin{align*}
S' = \set{A\in V_{\Sigma}\,\big\vert\,\rk A = \min\set{\ell,n}}
\end{align*}
is relative generic in~$V_{\Sigma}$. By Lemma~\ref{lem:How_do_feasible_perturbations_look_like},~$V_{\Sigma}$ is a real vector space and therefore convex. Further,~$S'$ is, in view of Lemma~\ref{lem:minors_are_polynomials}, relative Zariski open in~$V_{\Sigma}$. Hence, we can apply Lemma~\ref{lem:convex_reference_sets} and conclude that~$S'$ is relative generic in~$V_\Sigma$ if, and only if,~$S'$ is non-empty. Lemma~\ref{lem:How_do_feasible_perturbations_look_like} yields that
\begin{align*}
A' = \left[\begin{array}{c|c}
I_k & 0_{k\times (n-k)}\\\hline
0_{(\ell-k)\times\ell} & \begin{array}{cc}
I_{\min\set{\ell-k, n-k}} & 0\\
0 & 0
\end{array}
\end{array}\right]\in S'.
\end{align*}
This shows that~$S'$ is indeed non-empty and thus relative generic in~$V_\Sigma$. Especially,~$S'$ is a dense subset of~$V_\Sigma$. Thus, each neighbourhood of~$(E,Q)$ contains some~$(E,Q')$ so that~$PQ'T\in S'$ or, equivalently,~$(E,Q')\in S_{(i)}$. We conclude that~$S_{(i)}$ is indeed relative generic in~$V$.
\end{proof}

Using the results of Lemma~\ref{lem:Hamiltonian_pairs_full_rank_pointwise}, we can show that the rank of~$E^*Q$ is generically full.

\begin{Lemma}\label{lem:Hamiltonian_function_generically_full_rank}
Let~$\ell,n\in\mathbb{N}^*$. The set
\begin{align*}
S = \set{(E,Q)\in\mathbb{K}^{\ell\times n}\times\mathbb{K}^{\ell\times n}\,\big\vert\,\rk E^*Q = \min\set{\ell,n}}
\end{align*}
is relative generic in
\begin{align*}
V = \set{(E,Q)\in\mathbb{K}^{\ell\times n}\times\mathbb{K}^{\ell\times n}\,\big\vert\,E^*Q = Q^*E}.
\end{align*}
\end{Lemma}
\begin{proof} For brevity  put~$k = \min\set{\ell,n}$. Consider the set
\begin{align*}
\widetilde{V} = \set{(E,Q)\in\mathbb{K}^{\ell\times n}\times\mathbb{K}^{\ell\times n}\,\big\vert\,E^*Q = Q^*E,\ \rk E = k}.
\end{align*}
Using the notation of Lemma~\ref{lem:Hamiltonian_pairs_full_rank_pointwise}, we see rightforthly that
\begin{align*}
\widetilde{V} = S_{(ii)}\cap V.
\end{align*}
Hence Lemma~\ref{lem:Hamiltonian_pairs_full_rank_pointwise}\,(ii) yields that~$\widetilde{V}$ is a relative generic subset of~$V$. In view of Proposition~\ref{prop:properties_rel_gen}\,(e), it suffices to show that~$S$ is relative generic in the reference set~$\widetilde{V}$. Since~$\rk E^*Q\leq k$ for all~$E,Q\in\mathbb{K}^{\ell\times n}$,~$S\cap \widetilde{V}$ is in view of Lemma~\ref{lem:minors_are_polynomials} relative Zariski open in~$\widetilde{V}$. Therefore, it remains to show that~$S\cap\widetilde{V}$ is Euclidean dense in~$\widetilde{V}$. Let~$(E,Q)\in\widetilde{V}$.
In view of Lemma~\ref{lem:How_do_feasible_perturbations_look_like}, there are unitary (orthogonal) matrices~$P\in\mathbf{Gl}(\mathbb{K}^\ell)$ and~$T\in\mathbf{Gl}(\mathbb{K}^n)$, and some diagonal matrix~$\Sigma\in\mathbf{Gl}(\mathbb{R}^k)$ so that 
\begin{align*}
PET = \begin{bmatrix}
\Sigma & 0_{k\times (n-k)}\\
0_{(\ell-k)\times k} & 0_{(\ell-k)\times (n-k)}
\end{bmatrix}\quad\text{and}\quad PQ'T =\begin{bmatrix}
\widetilde{Q} & 0_{k\times(m-k)}\\
R_1 & R_2
\end{bmatrix} \in V_{\Sigma}.
\end{align*}
By~\eqref{eq:proof4},~$\rk E^*Q = k$ if, and only if, $\widetilde{Q}\in\mathbf{Gl}(\mathbb{K}^{k}).$
We show that the set
\begin{align*}
S' = \set{M\in V_\Sigma\,\Big\vert\, [M_{i,j}]_{i,j\in\mathbb{N}^*_{\leq k}}\in\mathbf{Gl}(\mathbb{K}^k)}
\end{align*}
is relative generic in~$V_\Sigma$. By Lemma~\ref{lem:minors_are_polynomials},~$S'$ is a relative Zariski-open subset of~$V_\Sigma$. Since~$V_\Sigma$ is convex,  Lemma~\ref{lem:convex_reference_sets} yields that it suffices to verify that~$S'$ is non-empty. The latter, however, holds true since
\begin{align*}
\begin{bmatrix}
I_k & 0_{k\times (n-k)}\\
0_{(\ell-k)\times k} & 0_{(\ell-k)\times (n-k)}
\end{bmatrix}\in S'.
\end{align*}
Thus~$S'$ is indeed relative generic in~$V_\Sigma$. Hence, each neighbourhood of~$(E,Q)$ contains some~$(E,Q')\in\widetilde{V}$ so that~$PQ'T\in S'$ or, equivalently,~$(E,Q')\in S\cap \widetilde{V}$. This proves the assertion.
\color{black}
\end{proof}

We have not yet taken into account that~$E^*Q$ is positive semidefinite for port-Hamiltonian systems. We adapt the proof of Lemma~\ref{lem:Hamiltonian_pairs_full_rank_pointwise} to show that, for fixed~$Q$,~$E$ has generically full rank.

\begin{Lemma}\label{lem:semidefinite_Hamiltonian_pair_pointwise_invertible}
Let~$\ell,n\in\mathbb{N}^*$ and~$Q\in\mathbb{K}^{\ell\times n}$. The set
\begin{align*}
S = \set{E\in\mathbb{K}^{\ell\times n}\,\big\vert\,\rk E = \min\set{\ell,n}}
\end{align*}
is relative generic in the set
\begin{align*}
V_Q = \set{E\in\mathbb{K}^{\ell\times n}\,\big\vert\, E^*Q = Q^*E\geq 0}.
\end{align*}
\end{Lemma}
\begin{proof}
It is rightforthly verified that~$V_Q$ is a convex set. Further,~$S$ is in view of Lemma~\ref{lem:minors_are_polynomials} Zariski-open. By Lemma~\ref{lem:convex_reference_sets},~$S$ is relative generic in~$V_Q$ if, and only if,~$S\cap V_Q$ is non-empty. In view of Lemma~\ref{lem:How_do_feasible_perturbations_look_like}, there are unitary matrices~$P\in\mathbf{Gl}(\mathbb{R}^\ell)$ and~$T\in\mathbf{Gl}(\mathbb{R}^n)$ and a diagonal matrix~$\Sigma\in\mathbf{Gl}(\mathbb{R}^k)$,~$k = \rk Q$, so that
\begin{align*}
PQT = \begin{bmatrix}
\Sigma & 0_{k\times (n-k)}\\
0_{(\ell-k)\times k} & 0_{(\ell-k)\times (n-k)}
\end{bmatrix}.
\end{align*}
In view of Lemma~\ref{lem:How_do_feasible_perturbations_look_like}, we find
\begin{align*}
P^*\left[\begin{array}{c|c}
\Sigma_k & 0_{k\times (n-k)}\\\hline
0_{(\ell-k)\times\ell} & \begin{array}{cc}
I_{\min\set{\ell-k,n-k}} & 0\\
0 & 0
\end{array}
\end{array}\right]T^*\in V_Q\cap S.
\end{align*}
This shows that~$S\cap V_Q$ is non-empty. Thus~$S$ is indeed relative generic in~$V_Q$.
\end{proof}

Lemma~\ref{lem:Hamiltonian_function_generically_full_rank} shows that in the case~$\ell\geq n$ the matrix~$E^*Q$ is generically invertible. We will use a similiar argument to show (essentially analogous to the proof of Lemma~\ref{lem:Hamiltonian_function_generically_full_rank}) that positive semidefinite matrices of the form~$E^*Q$ are generically positive definite.

\begin{Lemma}\label{lem:HUGE_simplification}
Let~$\ell\geq n\in\mathbb{N}^*$. The set
\begin{align*}
S = \set{(E,Q)\in\mathbb{K}^{\ell\times n}\times\mathbb{K}^{\ell\times n}\,\big\vert\,E^*Q = Q^*E>0}
\end{align*}
is relative generic in 
\begin{align*}
V = \set{(E,Q)\in\mathbb{K}^{\ell\times n}\times\mathbb{K}^{\ell\times n}\,\big\vert\,E^*Q = Q^*E\geq 0}.
\end{align*}
\end{Lemma}
\begin{proof}
A semi-definite matrix is not definite if, and only if, it has non-empty kernel. Thus, we have
\begin{align*}
V\setminus S = \set{(E,Q)\in \mathbb{K}^{\ell\times n}\times\mathbb{K}^{\ell\times n}\,\big\vert\,\det E^*Q = 0}\cap V.
\end{align*}
Since the determinant is a polynomial,~$V\setminus S$ is an algebraic set. Equivalently,~$V\cap S$ is Zariski-open. By Lemma~\ref{lem:rel_gen_char} and since~$S\subseteq V$, it remains to show that~$S$ is Euclidean dense in~$V$. Define the norm
\begin{align*}
\norm{\cdot}':\mathbb{K}^{\ell\times n}\times\mathbb{K}^{\ell\times n}\to\mathbb{R}_{\geq 0},\qquad (E,Q)\mapsto\max\set{\norm{E}_{\ell,n},\norm{Q}_{\ell,n}}
\end{align*}
for some operator norm~$\norm{\cdot}_{\ell,n}$ on~$\mathbb{K}^{\ell\times n}$. Since all norms on~$\mathbb{K}^{\ell\times n}\times\mathbb{K}^{\ell\times n}$ are equivalent, it suffices to prove density w.r.t.~$\norm{\cdot}'$. Let~$(E,Q)\in V$ and~$\varepsilon>0$. In view of Lemma~\ref{lem:semidefinite_Hamiltonian_pair_pointwise_invertible}, there is some~$H\in\mathbb{K}^{\ell\times n}$ with~$\norm{E-H}_{\ell,n}<\varepsilon$,~$\rk H = n$ and~$H^*Q = Q^*H\geq 0$. From the spectral theorem for symmetric matrices (see~\cite[Theorem 2.5.3]{HornJohn12}), we conclude that there is some orthogonal matrix~$O\in\mathbf{Gl}(\mathbb{K}^n)$ so that
\begin{align*}
O^*H^*QO = \begin{bmatrix}
\lambda_1\\
& \ddots\\
& & \lambda_n
\end{bmatrix}
\end{align*}
for some~$\lambda_1,\ldots,\lambda_n\in\mathbb{R}_{\geq 0}$. 
Without loss of generality we may assume that the~$\lambda_i$'s are ordered in such a manner that~$\lambda_1\geq \cdots\geq \lambda_k > 0 = \lambda_{k+1} = \cdots = \lambda_n$. Since~$H$ has full rank and~$\ell\geq n$, the columns of~$H^*$ contain a basis of~$\mathbb{K}^n$. Therefore, for each~$x\in\mathbb{K}^n$, there is some~$y\in\mathbb{K}^\ell$ with~$x = H^*y$. Especially, there is some~$\Delta\in\mathbb{K}^{\ell\times n}$ so that 
\begin{align*}
H^*\Delta = O\begin{bmatrix}
0_{k\times k}\\
& I_{n-k}
\end{bmatrix}O^*.
\end{align*}
Since the right-hand side of this equation is Hermitian,~$H^*\Delta = \Delta^* H$. Invoking unitarity (orthogonality) of~$O$, we have equivalently
\begin{align*}
O^*H^*\Delta O = \begin{bmatrix}
0_{k\times k}\\
& I_{n-k}
\end{bmatrix}.
\end{align*}
Therefore~$H^*(Q+\frac{\varepsilon}{\norm{\Delta}_{\ell,n}}\Delta)$ is Hermitian with
\begin{align*}
O^*H^*\left(Q+\frac{\varepsilon}{2\norm{\Delta}_{\ell,n}}\Delta\right)O = \begin{bmatrix}
\lambda_1\\
& \ddots\\
& & \lambda_k\\
& & & \frac{\varepsilon}{2\norm{\Delta}_{\ell,n}}\\
& & & & \ddots\\
& & & & & \frac{\varepsilon}{2\norm{\Delta}_{\ell,n}}
\end{bmatrix}>0.
\end{align*}
Hence, the assertion follows from
\begin{align*}
    H^*(Q+\tfrac{\varepsilon}{\norm{\Delta}_{\ell,n}}\Delta) = (Q+\tfrac{\varepsilon}{\norm{\Delta}_{\ell,n}}\Delta)^*H 
    >0~\text{ and}~\norm{(E,Q)-(H,Q+\tfrac{\varepsilon}{\norm{\Delta}_{\ell,n}}\Delta)}'<\varepsilon.
\end{align*}
\end{proof}

Before we prove that port-Hamiltonian descriptor systems are generically controllable, however, we need an auxiliary result, where we show that the rank of the block matrices that appear in the algebraic criteria in Proposition~\ref{Def-Prop:Contr} is generically full. Recall that our reference sets of port-Hamiltonian matrix tuples are defined as
\begin{align*}
\Sigma_{\ell,n,m}^{H} &= \set{(E,J,Q,B)\in\mathbb{K}^{\ell\times n}\times\mathbb{K}^{\ell\times \ell}\times\mathbb{K}^{\ell\times n}\times\mathbb{K}^{\ell\times m}\,\big\vert\,J=-J^*,E^*Q = Q^*E\geq 0},\\[2ex]
\Sigma_{\ell,n,m}^{sdH} &= \set{(E,J,R,Q,B)\,\big\vert\,J=-J^*,R = R^*\geq 0,E^*Q = Q^*E\geq 0}.
\end{align*}

\begin{Proposition}\label{lem:main_result_1}
Let~$\ell,n,m\in\mathbb{N}^*$. The sets
\begin{enumerate}[(a)]
\item~$S_{(a)} := \set{(E,J,Q,B)\,\big\vert\,\rk[E,B] = \min\set{\ell,n+m}}$
\item~$S_{(b)} := \set{(E,J,Q,B)\,\big\vert\,\rk[E,JQ,B] = \min\set{\ell,2n+m}}$
\item~$S_{(c)} := \set{(E,J,Q,B)\,\big\vert\,\rk_{\mathbb{K}(x)}[xE-JQ,B] = \min\set{\ell,n+m}}$
\item~$S_{(d)} := \set{(E,J,Q,B)\,\left\vert\,\begin{array}{l}\forall Z\in\mathbb{K}^{n-\rk E}~\text{with}~\im Z = \ker E\\: \rk[E,JQZ,B] = \min\set{\ell,n+m}\end{array}\right.}$
\end{enumerate}
are relative generic in~$\Sigma_{\ell,n,m}^H$. The sets
\begin{enumerate}[(a)]
\setcounter{enumi}{4}
\item~$S_{(e)} := \set{(E,J,R,Q,B)\,\big\vert\,\rk[E,B] = \min\set{\ell,n+m}}$
\item~$S_{(f)} := \set{(E,J,R,Q,B)\,\big\vert\,\rk[E,(J-R)Q,B] = \min\set{\ell,2n+m}}$
\item~$S_{(g)} := \set{(E,J,R,Q,B)\,\big\vert\,\rk_{\mathbb{R}(x)}[xE-(J-R)Q,B] = \min\set{\ell,2n+m}}$
\item~$S_{(h)} := \set{(E,J,R,Q,B)\,\left\vert\,\begin{array}{l}\forall Z\in\mathbb{K}^{n\times (n-\rk E)}~\text{with}~\ker E = \im Z\\: \rk[E,(J-R)QZ,B] = \min\set{\ell,n+m}\end{array}\right.}$
\end{enumerate}
are relative generic in~$\Sigma_{\ell,n,m}^{sdH}$ and~$\Sigma_{\ell,n,m}^{dH}$. The sets
\begin{enumerate}[(a)]
\setcounter{enumi}{8}
\item~$S_{(i)} := \set{(E,J,Q,B)\,\big\vert\,\forall\lambda\in\mathbb{C}:\rk_{\mathbb{C}}[\lambda E-JQ,B] = \min\set{\ell,n+m}}$
\item~$S_{(j)} := \set{(E,J,R,Q,B)\,\big\vert\,\forall\lambda\in\mathbb{C}:\rk_{\mathbb{C}}[\lambda E-(J-R)Q,B] = \min\set{\ell,n+m}}$
\end{enumerate}
are relative generic in~$\Sigma_{\ell,n,m}^H$ and~$\Sigma_{\ell,n,m}^{sdH}$, resp., if, and only if,~$\ell\neq n+m$; otherwise their complement is relative generic in the respective reference set.
\end{Proposition}

\begin{proof} This proof is split into the two cases~$\ell\geq n$ and~$\ell <n$.
\\

\noindent
Case~$\ell\geq n$:  \ 
We find convenient reference sets to simplify our calculations. 
Note that
\begin{align*}
\Sigma_{\ell,n,m}^{H}\mathbb{C}ong\set{(E,Q)\in\mathbb{K}^{\ell\times n}\times \mathbb{K}^{\ell\times n}\,\big\vert\,E^*Q = Q^*E\geq 0}\times\set{J\in\mathbb{K}^{\ell\times \ell}\,\big\vert\,J = -J^*}\times\mathbb{K}^{\ell\times m}.
\end{align*}
In Lemma~\ref{lem:HUGE_simplification} we have shown that~$\set{(E,Q)\in\mathbb{K}^{\ell\times n}\times \mathbb{K}^{\ell\times n}\,\big\vert\,E^*Q = Q^*E > 0}$ is relative generic in~$\set{(E,Q)\in\mathbb{K}^{\ell\times n}\times \mathbb{K}^{\ell\times n}\,\big\vert\,E^*Q = Q^*E\geq 0}$. In view of Proposition~\ref{prop:properties_rel_gen}\,(d), we conclude that
\begin{align*}
V_1 := \set{(E,J,Q,B)\in\Sigma_{\ell,n,m}^{H}\,\big\vert\, E^*Q = Q^*E>0}\subseteq\Sigma_{\ell,n,m}^{H}
\end{align*}
is a relative generic subset of~$\Sigma_{\ell,n,m}^H$. Thus, in view of Proposition~\ref{prop:properties_rel_gen}\,(e), relative genericity in~$\Sigma_{\ell,n,m}^H$ and~$V_1$ are equivalent. By definition of~$\Sigma_{\ell,n,m}^H$ and continuity of the spectrum,~$V_1$ is a relatively open subset of
\begin{align*}
V_1' := \set{(E,J,Q,B)\in\mathbb{K}^{\ell\times n}\times\mathbb{K}^{\ell\times \ell}\times\mathbb{K}^{\ell\times n}\times\mathbb{K}^{\ell\times m}\,\big\vert\,E^*Q = Q^*E, J^* = -J}.
\end{align*}
By Lemma~\ref{lem:rel_open_reference_set} it suffices to prove the assertion for the reference set~$V_1'$ instead of~$\Sigma_{\ell,n,m}^H$. For~$\Sigma_{\ell,n,m}^{sdH}$ we have
\begin{align*}
\Sigma_{\ell,n,m}^{sdH} \mathbb{C}ong & \Sigma_{\ell,n,m}^{H}\times\set{R\in\mathbb{K}^{\ell\times \ell}\,\big\vert\,R^* = R\geq 0}.
\end{align*}
Since~$V_1$ is a relative generic subset of~$\Sigma_{\ell,n,m}^H$, we conclude from Lemma~\ref{lem:definite_is_generic} and Proposition~\ref{prop:properties_rel_gen}\,(d) that 
\begin{align*}
\set{(E,J,R,Q,B)\,\big\vert\,E^*Q = Q^*E > 0, J^* = -J, R^* = R > 0}\mathbb{C}ong V_1\times\set{R\,\big\vert\,R^* = R > 0}
\end{align*}
is a relative generic subset of~$\Sigma_{\ell,n,m}^{sdH}$. 
By continuity of the spectrum, we conclude with Proposition~\ref{prop:properties_rel_gen}\,(e) and Lemma~\ref{lem:rel_open_reference_set} that it suffices to consider the reference set
\begin{align*}
V_2' := & \set{(E,J,R,Q,B)\,\big\vert\,E^*Q = Q^*E, J^* = -J, R^* = R} \mathbb{C}ong V_1'\times\set{R\in\mathbb{K}^{\ell\times \ell}\,\big\vert\,R^* = R}
\end{align*}
instead of~$\Sigma_{\ell,n,m}^{sdH}$.
We consider therefore, without loss of generality, the reference set~$V_1'$ instead of~$\Sigma_{\ell,n,m}^{H}$, and the reference set~$V_2'$ instead of~$\Sigma_{\ell,n,m}^{sdH}$.
\\

\noindent
(a)\quad 
Since all minors are polynomials,~$S_{(a)}$ is in view of Lemma~\ref{lem:minors_are_polynomials} Zariski-open. In view of Lemma~\ref{lem:rel_gen_char}, it remains to prove that~$S_{(a)}\cap V_1'$ is Euclidean dense in~$V_1'$. 
Let~$(E,J,Q,B)\in V_1'$ and~$\varepsilon>0$. Since relative generic sets are especially Euclidean dense, see Lemma~\ref{prop:properties_rel_gen}\,(a), Lemma~\ref{lem:Hamiltonian_pairs_full_rank_pointwise}\,(ii) yields that there are some~$\widetilde{E},\widetilde{Q}\in\mathbb{K}^{\ell\times n}$ so that~$\widetilde{E}^*\widetilde{Q} = \widetilde{Q}^*\widetilde{E}$,~$\rk \widetilde{E} = \min\set{\ell,n} = n$, and
\begin{align*}
\max\set{\norm{E-\widetilde{E}}_{\ell,n},\norm{Q-\widetilde{Q}}_{\ell,n}}<\varepsilon.
\end{align*}
Consider the set
\begin{align*}
S_{(a)}(\widetilde{E}) := \set{B\in\mathbb{K}^{\ell\times m}\,\big\vert\,\rk[\widetilde{E},B] = \min\set{\ell,n+m}}.
\end{align*}
By Lemma~\ref{lem:minors_are_polynomials},~$S_{(a)}(\widetilde{E})$ is Zariski-open. Furthermore, in view of the basis completion lemma and since~$\rk\widetilde{E} = n$, it can be readily seen that~$S_{(a)}(\widetilde{E})\neq\emptyset$. Thus Lemma~\ref{lem:gen_characterisation} yields that~$S_{(a)}(\widetilde{E})$ is generic. Especially, we conclude that there is some~$\widetilde{B}\in\mathbb{K}^{\ell\times m}$ so that~$\rk[\widetilde{E},\widetilde{B}] = \min\set{\ell,n+m}$ and~$\norm{B-\widetilde{B}}_{\ell,m}<\varepsilon$. Thus, we have found some~$(\widetilde{E},J,\widetilde{Q},\widetilde{B})\in S_{(a)}\cap V_1'$ with
\begin{align*}
\max\set{\norm{E-\widetilde{E}}_{\ell,n},\norm{Q-\widetilde{Q}}_{\ell,n},\norm{B-\widetilde{B}}_{\ell,m}}<\varepsilon.
\end{align*}
This shows that~$S_{(a)}\cap V_1'$ is indeed Euclidean dense and therefore relative generic in~$V_1'$.
\\
\color{black}

\noindent
(b)\quad 
We distinguish further between~$\ell\leq n+m$ and~$\ell>n+m$. 
\\
If~$\ell\leq n+m$, then we have~$\ell = \min\set{\ell,n+m} = \min\set{\ell,2n+m}$ and thus the inclusion~$S_{(a)}\subseteq S_{(b)}$ holds true. Since we have already shown that~$S_{(a)}$ is relative generic in~$V_1'$, Proposition~\ref{prop:properties_rel_gen}\,(b) yields that~$S_{(b)}$ is relative generic in~$V_1'$.
\\ 
Let~$\ell>n+m$.
In view of Lemma~\ref{lem:minors_are_polynomials},~$S_{(b)}$ is Zariski open. By Lemma~\ref{lem:rel_gen_char}, it remains to prove that~$S_{(b)}\cap V_1'$ is Euclidean dense in~$V_1'$.
Let~$(E,J,Q,B)\in V_1'$ and~$\varepsilon>0$. Analogously to (a), Lemma~\ref{lem:Hamiltonian_pairs_full_rank_pointwise}\,(ii) implies the existence of some~$\widetilde{E},\widetilde{Q}\in\mathbb{K}^{\ell\times n}$ so that~$\rk\widetilde E = n$,~$\widetilde{E}^*\widetilde{Q} = \widetilde{Q}^*\widetilde{E}$ and
\begin{align*}
\max\set{\| {E-\widetilde{E}}\|_{\ell,n}, \|{Q-\widetilde{Q}}\|_{\ell,n}}< \varepsilon/{2}.
\end{align*}
Consider the convex reference set
\begin{align*}
V_{(b)}(\widetilde{E}) = \set{(J,Q,B)\in\mathbb{K}^{\ell\times\ell}\times\mathbb{K}^{\ell\times n}\times\mathbb{K}^{\ell\times m}\,\big\vert\,J = -J^*, \widetilde{E}^*Q = Q^*\widetilde{E}}
\end{align*}
and the set
\begin{align*}
S_{(b)}(\widetilde{E}) = \set{(J,Q,B)\in V_{(b)}(\widetilde{E})\,\big\vert\,\rk[\widetilde{E},JQ,B] = \min\set{\ell,2n+m}}.
\end{align*}
We show that~$S_{(b)}(\widetilde{E})$ is relative generic in~$V_{(b)}(\widetilde{E})$. In view of Lemma~\ref{lem:minors_are_polynomials},~$S_{(b)}(\widetilde{E})\subseteq V_{(b)}(\widetilde{E})$ is relative Zariski-open in~$V_{(b)}(\widetilde{E})$. Since~$V_{(b)}(\widetilde{E})$ is convex, Lemma~\ref{lem:convex_reference_sets} yields that it suffices to prove that~$S_{(b)}(\widetilde{E})\neq\emptyset$. In view of Lemma~\ref{lem:How_do_feasible_perturbations_look_like}, there are unitary (orthogonal) matrices~$P\in\mathbf{Gl}(\mathbb{K}^\ell)$,~$T\in\mathbf{Gl}(\mathbb{K}^n)$ and some diagonal matrix~$\Sigma\in\mathbb{R}^{n\times n}$ so that
\begin{align*}
P\widetilde{E}T = \begin{bmatrix}
\Sigma\\
0_{(\ell-n)\times n}
\end{bmatrix}.
\end{align*}
Choose the skew-symmetric matrix~$J'\in\mathbb{K}^{\ell\times \ell}$ with
\begin{align*}
\forall i,j\in\mathbb{N}^*_{\leq\ell}: J'_{i,j} = \begin{cases}
1, & i = n+m+j,\\
-1, & j = n+m+i,\\
0, & \text{else},
\end{cases}
\end{align*}
and put
\begin{align*}
\widehat{J}' = P^*J'P,\quad \widehat{Q}' = \widetilde{E},\quad\text{and}\quad\widehat{B}' := P^*\begin{bmatrix}
0_{n\times m}\\
I_m\\
0_{(\ell-n-m)\times m}
\end{bmatrix}.
\end{align*}
Let~$\kappa :=\min\set{\ell-n-m,n}$. Since~$P$ and~$T$ are unitary and~$\Sigma\in\mathbb{K}^{n\times n}$ is an invertible diagonal matrix, we have
\begin{align*}
\rk [\widetilde{E},\widehat{J}'\widehat{Q}',\widehat{B}'] & 
= \rk  P [\widetilde{E},\widehat{J}'\widehat{Q}',\widehat{B}']\begin{bmatrix}
T\\ & T\\ & & I_m
\end{bmatrix}\\[1ex]
& = \rk[P\widetilde{E}T,P\widehat{J}'P^*P\widehat{Q}'T,P\widehat{B}']
\\[1ex]
& = \rk\left[\begin{array}{c|c|c}
\Sigma & 0_{n\times n} & 0_{n\times m}\\
0_{m\times n} & 0_{m\times n} & I_m\\\hline
0_{\kappa\times n} & \begin{array}{cccccccc}
\Sigma_{1,1} & \star & \cdots & \star & \star & \cdots & \star\\
 & \ddots & \ddots & \vdots & \ddots & \ddots & \vdots\\
 & & \Sigma_{\kappa-1,\kappa-1} & \star & \star & \cdots & \star\\
 & & & \Sigma_{\kappa,\kappa} & \star & \cdots & \star
\end{array} & 0_{\kappa\times m}\\
0_{(\ell-n-m-\kappa)\times n} & 0_{(\ell-n-m-\kappa)\times n} & 
0_{(\ell-n-m-\kappa)\times n}
\end{array}\right]\\
& = n+m+\kappa\\
& = \min\set{\ell,2n+m}.
\end{align*}
This shows that~$(\widehat{J}',\widehat{Q}',\widehat{B}')\in S_{(b)}(\widetilde{E})$ and thus the latter set is indeed non-empty. We conclude that~$S_{(b)}(\widetilde{E})$ is relative generic in~$V_{(b)}(\widetilde{E})$. Hence, there are~$(\widehat{J},\widehat{Q},\widehat{B})\in S_{(b)}(\widetilde{E})$ so that
\begin{align*}
\max\set{\norm{\widehat{J}-J}_{\ell,\ell},\norm{\widehat{Q}-\widetilde{Q}}_{\ell,n},\norm{\widehat{B}-B}_{\ell,m}}<\frac{\varepsilon}{2}.
\end{align*}
From the definition of~$S_{(b)}(\widetilde{E})$, we conclude that~$(\widetilde{E},\widehat{J},\widehat{Q},\widehat{B})\in S_{(b)}\cap V_1'$, and the triangle inequality yields
\begin{align*}
\max\set{\norm{\widetilde{E}-E}_{\ell,n},\norm{\widehat{J}-J}_{\ell,\ell},\norm{\widehat{Q}-Q}_{\ell,n},\norm{\widehat{B}-B}_{\ell,m}}<\varepsilon.
\end{align*}
This shows that~$S_{(b)}\cap V_1'$ is indeed Euclidean dense in~$V_1'$, and therefore relative generic in~$V_1'$.
\\

\noindent
(c)\quad 
Analogously to the proof of Lemma~\ref{lem:minors_are_polynomials}, a simple application of the Leibniz formula for the determinant yields with Lemma~\ref{lem:important_property} that the implication
\begin{align*}
\rk_{\mathbb{K}}[E,B] = \min\set{\ell,n+m}\implies \forall A\in\mathbb{K}^{\ell\times n}: \rk_{\mathbb{K}(x)}[xE-A,B] = \min\set{\ell,n+m}
\end{align*}
holds true for all~$E\in\mathbb{K}^{\ell\times n}$ and~$B\in\mathbb{K}^{\ell\times m}$. Thus, we have~$S_{(a)}\subseteq S_{(c)}$. Since we have proven that~$S_{(a)}$ is relative generic in~$V_1'$, we conclude from Proposition~\ref{prop:properties_rel_gen}\,(b) that~$S_{(c)}$ is relative generic in~$V_1'$.
\\

\noindent
(d)\quad 
Since~$\rk E = n$ yields~$\ker E = \set{0}$, the inclusion
\begin{align*}
S_{(a)}\cap \set{(E,J,Q,B)\,\big\vert\,\rk E = n}\subseteq S_{(d)}
\end{align*}
holds true. Since~$S_{(a)}$ is relative generic in~$V_1'$ by (a) of the present proposition, we conclude from Lemma~\ref{lem:Hamiltonian_pairs_full_rank_pointwise}\,(ii) and Proposition~\ref{prop:properties_rel_gen}\,(c) that the intersection~$S_{(a)}\cap \set{(E,J,Q,B)\,\big\vert\,\rk E = n}$ is relative generic in~$V_1'$. Thus, Proposition~\ref{prop:properties_rel_gen}\,(b) yields that~$S_{(d)}$ is relative generic in~$V_1'$.
\\

\noindent
The proof of (e)--(h) is analogous to the proof of~(a)--(d): While the proof of~(a), (c) and~(d) does not depend on~$J$, their counterparts~(e), (g) and~(h) follow analogously. To prove~(f), we can simply add a summand
$\set{R\in\mathbb{K}^{\ell\times \ell}\,\big\vert R = R^*}$ to the real vector space~$V_{(b)}(\widetilde{E})$ in the proof of~(b) and proceed analogously.
\\

\noindent
(i)\quad 
We proceed with the proof of (i) 
in steps~(i$_1$)--(i$_5$).

\noindent
\textsc{Step}~(i$_1$): \quad 
We show that if~$\ell = n+m$, then~$S_{(i)}^c$ is relative generic in~$V_1'$.
\\
 In view of Laplace's expansion formula, for all~$(E,J,Q,B)\in S_{(a)}$,
\begin{align*}
\det [xE-JQ,B] = x^n\underbrace{\det[E,B]}_{\neq 0} + \text{lower~order~terms}.
\end{align*}
Therefore, the fundamental theorem of algebra yields that there is some~$\lambda\in\mathbb{C}$ so that~$\det [\lambda E-JQ,B] = 0$. This shows that~$S_{(a)}\subseteq S_{(i)}^c$. Since~$S_{(a)}$ is relative generic in~$V_1'$, Proposition~\ref{prop:properties_rel_gen}\,(b) yields that~$S_{(i)}^c$ is relative generic in~$V_{1}'$. Especially~$S_{(i)}$ is not relative generic in~$V_1'$.
\\

\noindent
\textsc{Step}~(i$_2$): \quad
Let~$\ell\neq n+m$ and~$\ell\geq n$. We construct an algebraic set~$\mathbb{V}$ so that~$V_1'\setminus S_{(i)}\subseteq \mathbb{V}$. 
\\
This construction is analogous to the construction in the proof of~\cite[Proposition B.8]{IlchKirc21}, where a similiar result for matrices without structural constraints is shown. Let~$\widetilde{M}_1,\ldots,\widetilde{M}_q$,~$q\in\mathbb{N}$ be all minors of order~$d := \min\set{\ell,n+m}$ w.r.t.~$\mathbb{K}[x]^{\ell\times (n+m)}$ and put, for all~$i\in\mathbb{N}^*_{\leq q}$ and
$P\in\mathbf{Gl}(\mathbb{K}^\ell)$,
\begin{align*}
M_i^{P}:\mathbb{K}^{\ell\times n}\times\mathbb{K}^{\ell\times \ell}\times\mathbb{K}^{\ell\times n}\times\mathbb{K}^{\ell\times m} &\to\mathbb{K}[x],\\ (E,J,Q,B) & \mapsto \widetilde{M}_i\left(P[xE-JQ,B]\right).
\end{align*}
Since the rank of a matrix is invariant under regular transformations from the left and from the right, we have by Lemma~\ref{lem:important_property}~$(E,J,Q,B)\in S_{(i)}$ if, and only if,
\begin{align*}
\forall\lambda\in\mathbb{C}~\exists i\in\mathbb{N}^*_{\leq q}~\exists P\in\mathbf{Gl}(\mathbb{K}^\ell): M_i^{P}(E,J,Q,B)(\lambda)\neq 0.
\end{align*}
Define the maximal degree of~$M_i^{P}$ as
\begin{align*}
\gamma_i^{P} := \max\set{\deg M_i^{P}(E,J,Q,R)\,\big\vert\,(E,J,Q,B)\in\mathbb{K}^{\ell\times n}\times\mathbb{K}^{\ell\times \ell}\times\mathbb{K}^{\ell\times n}\times\mathbb{K}^{\ell\times m}}\geq 0.
\end{align*}
In passing, note that the maximal degree~$\gamma_i^{P}$ is attained at~$(E,J,Q,B)$ if, and only if
\begin{align*}
\deg M_i^{P}(E,0,0,B) = \deg\widetilde{M}_i([xPET,B]) = \gamma_i.
\end{align*}
Since~$(E,0,0,B)\in V_1'$ for all~$E\in\mathbb{K}^{\ell\times n}$ and~$B\in\mathbb{K}^{\ell\times m}$,
\begin{align*}
\gamma_i^{P} = \max\set{\deg M_i^{P}(E,J,Q,R)\,\big\vert\,(E,J,Q,B)\in V_1'}
\end{align*}
and, for all~$\mathcal{S}\subseteq\mathbb{K}^{\ell\times n}\times\mathbb{K}^{\ell\times \ell}\times\mathbb{K}^{\ell\times n}\times\mathbb{K}^{\ell\times m}$,
\begin{align*}
\set{(E,J,Q,B)\in\mathcal{S}\,\big\vert\,\deg M_i^{P}(E,J,Q,B) = \gamma_i} = \set{(E,J,Q,B)\in\mathcal{S}\,\big\vert\, M_i^{P}(E,0,0,B)\neq 0}.
\end{align*}
By Lemma~\ref{lem:minors_are_polynomials}, there are polynomials~$M_{i,k}^{P}\in\mathbb{K}[x_1,\ldots,x_{\ell(2n+m+\ell)}]$, $k \in\set{ 0,\ldots,\gamma_i^{P}}$,~$i\in\mathbb{N}^*_{\leq q}$, so that each~$M_i^{P}$ has the representation
\begin{align*}
\forall (E,J,Q,B)\in\mathbb{K}^{\ell\times n}\times\mathbb{K}^{\ell\times \ell}\times\mathbb{K}^{\ell\times n}\times\mathbb{K}^{\ell\times m}: M_i(E,J,Q,B) = \sum_{k = 0}^{\gamma_i} M_{i,k}^{P}(E,J,Q,B) x^k.
\end{align*}
The~$M_{i,k}^{P}$,~$k \in \set{0,\ldots,\gamma_{i}^{P}}$, which are multivariate polynomials in the entries of~$E,J,Q,B$, are the coefficients of~$M_i^{P}$ in the monomial basis of the univariate polynomials. Using these representations, define, for all~$i,j\in\mathbb{N}^*_{\leq q}$ and~$P\in\mathbf{Gl}(\mathbb{K}^\ell)$, the polynomials
\begin{align}\label{eq:Sylvester_operator}
p_{i,j}^{P}(\cdot)=\det\underbrace{\left[\begin{array}{ccccc|ccccc}
M_{i,0}^{P}(\cdot) & & & & & M_{j,0}^{P}(\cdot) & & &\\
M_{i,1}^{P}(\cdot) & M_{i,0}^{P}(\cdot) & & & & M_{j,1}^{P}(\cdot) & \cdot &\\
\cdot & \cdot & & & & \cdot & \cdot & \cdot & \\
\cdot & \cdot & \cdot & & & \cdot & \cdot & \cdot & M_{j,0}^{P}(\cdot) \\
M_{i,{\gamma_i}}^{P}(\cdot) & M_{i,{{\gamma_i}-1}}^{P}(\cdot) & \cdot & \cdot & & \cdot & \cdot & \cdot & \cdot\\
 & M_{i,{\gamma_i}}^{P}(\cdot) & \cdot & \cdot & & \cdot & \cdot & \cdot & \cdot\\
 & & \cdot & \cdot & M_{i,0}^{P}(\cdot) & M_{j,{{\gamma_j}-1}}^{P}(\cdot) & \cdot & \cdot & \cdot\\
& & \cdot & \cdot & \cdot & M_{j,{{\gamma_j}}}^{P}(\cdot) & \cdot & \cdot & \cdot & \\
& & \cdot & \cdot & \cdot & & \cdot & \cdot & \cdot \\
& & & \cdot & \cdot & & & \cdot & \cdot\\
& & & & M_{i,{\gamma_i}}^{P}(\cdot) & & & &  M_{j,{\gamma_j}}^{P}(\cdot)
\end{array}\right]}_{\in\mathbb{K}[x_1,\ldots,x_{\ell(2n+m+\ell)}]^{(\gamma_i^{P}+\gamma_j^{P})\times (\gamma_i^{P}+\gamma_j^{P})}}.
\end{align}
If ~$M_{i,\gamma_i}^{P}(E,J,Q,B)\neq 0$ and~$M_{j,\gamma_j}^{P}(E,J,Q,B)\neq 0$, then~$p_{i,j}^{P}(E,J,Q,B)$ is the Sylvester resultant (see Lemma~\ref{Def:Resultant}) of the polynomials~$M_i(E,J,Q,B)$ and~$M_j(E,J,Q,B)$. Invoking Lemma~\ref{Def:Resultant} and Laplace's expansion formula~\cite[p.\,203]{Fisc05b}, we have, for all~$\mathfrak{E} = (E,J,Q,B)\in\mathbb{K}^{\ell\times n}\times\mathbb{K}^{\ell\times \ell}\times\mathbb{K}^{\ell\times n}\times\mathbb{K}^{\ell\times m}$, the equivalence
\begin{align}\label{eq:properties_p_polynomials}
p_{i,j}^{P}(\mathfrak{E}) = 0 \iff \big[ M_{i,\gamma_i}^{P}(\mathfrak{E}) = M_{j,\gamma_j}^{P}(\mathfrak{E}) = 0\ \vee\ M_i(\mathfrak{E}),M_j(\mathfrak{E})~\text{not~coprime}\big].
\end{align}
Thus, Lemma~\ref{lem:important_property} yields that the following chain of implications holds true for all~$(E,J,Q,B)\in\mathbb{K}^{\ell\times n}\times\mathbb{K}^{\ell\times \ell}\times\mathbb{K}^{\ell\times n}\times\mathbb{K}^{\ell\times m}$:
\begin{align*}
& \exists P\in\mathbf{Gl}(\mathbb{K}^\ell)~\exists i,j\in\mathbb{N}^*_{\leq q}: p_{i,j}^{P}(E,J,Q,B)\neq 0\\
& \implies \exists i,j\in\mathbb{N}^*_{\leq q}: M_i^{P}(E,J,Q,B), M_j^{P}(E,J,Q,B)~\text{are~coprime}\\
& \implies \forall\lambda\in\mathbb{C}~\exists\iota\in\mathbb{N}^*_{\leq q}: \widetilde{M}_\iota[\lambda PET-PJP^*P^{-*}QT,B] = M_\iota^{P}(E,J,Q,B)(\lambda) \neq 0\\
& \implies \forall\lambda\in\mathbb{C}: \rk_{\mathbb{C}}[\lambda E-JQ,B]) = d.
\end{align*}

\noindent Let~$I\subseteq\mathbb{R}[x_1,\ldots,x_{\ell(2n+m+\ell)}]$ be the ideal generated by the set
\begin{align*}
\set{p_{i,j}^{P}\,\big\vert\,i,j\in\mathbb{N}^*_{\leq q}, P\in\mathbf{Gl}(\mathbb{K}^\ell)}.
\end{align*}
Since multivariable polynomials are a commutative ring, the ideal~$I$ has the form
\begin{align*}
I = \set{\left.\sum_{\alpha = 1}^k r_\alpha p_{i_\alpha,j_\alpha}^{P_\alpha}\,\right\vert\,\begin{array}{l}
k\in\mathbb{N}_0, i_\alpha,j_\alpha\in\mathbb{N}^*_{\leq q},r_\alpha\in\mathbb{R}[x_1,\ldots,x_{\ell(2n+m+\ell)}],\\
P_\alpha\in\mathbf{Gl}(\mathbb{K}^{\ell}),\alpha\in\mathbb{N}^*_{\leq k}
\end{array}},
\end{align*}
see~\cite[p.\,538]{BirkMacl88}. Define
\begin{align*}
\mathbb{V} := \set{(E,J,Q,B)\,\big\vert\,\forall p\in I: p(E,J,Q,B) = 0}
\end{align*}
as the algebraic set generated by the ideal~$I$, see Lemma~\ref{lem:Zariski_und_ideale}. Due to our chain of implications after~\eqref{eq:properties_p_polynomials}, we have
\begin{align*}
(E,J,Q,B)\in\mathbb{V}^c & \implies\exists P\in\mathbf{Gl}(\mathbb{K}^\ell)~\exists i,j\in\mathbb{N}^*_{\leq q}: p_{i,j}^{P}(E,J,Q,B)\neq 0\\
& \implies (E,J,Q,B)\in S_{(i)};
\end{align*}
hence, ~$\mathbb{V}^c\cap V_1'\subseteq S_{(i)}\cap V_1'$ or, equivalently,~$V_1'\cap S_{(i)}^c\subseteq \mathbb{V}\cap V_1'$.
\\
\color{black}

\noindent
\textsc{Step}~(i$_3$): \quad
We show that the set~$\mathbb{V}^c\cap V_1'$ is Euclidean dense in~$V_1'$.
\\
 The arguments used are similiar to those in the proof of~\cite[Proposition B.8]{IlchKirc21}. However~\cite[Proposition B.8]{IlchKirc21} deals with the unrestrained case; a more involved proof has to obey the reference set~$V_1'$. Let~$(E,J,Q,B)\in V_1'$ and~$\varepsilon>0$. By Lemma~\ref{lem:Hamiltonian_pairs_full_rank_pointwise}, there is some~$\widetilde{E},\widetilde{Q}\in\mathbb{K}^{\ell\times n}$ so that~$\rk\widetilde{E} = n$,~$\widetilde{E}^*\widetilde{Q} = \widetilde{Q}^*\widetilde{E}$ and
\begin{align*}
\max\set{\norm{E-\widetilde{E}}_{\ell,n},\norm{Q-\widetilde{Q}}_{\ell,n}}<\frac{\varepsilon}{2}.
\end{align*}
Consider the set
\begin{align*}
V_{(i)}(\widetilde{E}) := \set{(J,Q,B)\in\mathbb{K}^{\ell\times \ell}\times\mathbb{K}^{\ell\times n}\times\mathbb{K}^{\ell\times m}\,\big\vert\,J^* = -J, \widetilde{E}^*Q = Q^*\widetilde{E}}.
\end{align*}
Assume, for a moment, that we had already proven that the set
\begin{align*}
S_{(i)}(\widetilde{E}) := \set{(J,Q,B)\in V_{(i)}(\widetilde{E})\,\big\vert\,(\widetilde{E},J,Q,B)\in\mathbb{V}^c}
\end{align*} 
is relative generic in~$V_{(i)}(\widetilde{E})$. Since~$(J,\widetilde{Q},B)\in V_{(i)}(\widetilde{E})$, we find some~$(\widehat{J},\widehat{Q},\widehat{B})\in S_{(i)}(\widetilde{E})$ so that
\begin{align*}
\max\set{\norm{\widehat{J}-J}_{\ell,\ell},\norm{\widehat{Q}-\widetilde{Q}}_{\ell,n},\norm{\widehat{B}-B}_{\ell,m}}<\frac{\varepsilon}{2}.
\end{align*}
By definition of~$S_{(i)}(\widetilde{E})$ and~$V_{(i)}(\widetilde{E})$, we have~$(\widetilde{E},\widehat{J},\widehat{Q},\widehat{B})\in \mathbb{V}^c\cap V_1'$ so that
\begin{align*}
\max\set{\norm{E-\widetilde{E}}_{\ell,n},\norm{\widehat{J}-J}_{\ell,\ell},\norm{\widehat{Q}-Q}_{\ell,n},\norm{\widehat{B}-B}_{\ell,m}}<\varepsilon.
\end{align*}
This shows that~$\mathbb{V}^c\cap V_1'$ is indeed Euclidean dense in~$V_1'$. It remains to prove that~$S_{(i)}(\widetilde{E})$ is relative generic in~$V_{(i)}(\widetilde{E})$. The latter is convex and hence it remains, in view of Lemma~\ref{lem:convex_reference_sets}, to show that~$S_{(i)}(\widetilde{E})$ is non-empty. 
As in the proof of~\cite[Proposition B.8]{IlchKirc21}, we distinguish between the two cases~$\ell<n+m$ and~$\ell>n+m$; the case~$\ell = n+m$ was dealt with in
Step~(i$_1$).
\\

\noindent
\textsc{Step}~(i$_4$):\quad
 Let~$\ell<n+m$.
 \\
  Since~$\ell\geq n$, there is some~$T\in\mathbf{Gl}(\mathbb{K}^\ell)$ so that
\begin{align}\label{eq:block_decomposition_E}
T\widetilde{E} = \begin{bmatrix}
I_n\\ 0_{(\ell-n)\times n}
\end{bmatrix}.
\end{align}
Since~$T$ is regular we have, for all~$\lambda\in\mathbb{C}$ and~$(J,Q,B)\in V_{(i)}(\widetilde{E})$,
\begin{align*}
\rk[\lambda \widetilde{E}-JQ,B] = \rk T[\lambda \widetilde{E}-JQ,B] = \rk[\lambda T\widetilde{E}-TJT^*T^{-*}Q,TB].
\end{align*}
Moreover,~$Q\in\mathbb{K}^{\ell\times n}$ fullfills~$\widetilde{E}^*Q$ if, and only if,~$(T\widetilde{E})^*(T^{-*}Q) = (T^{-*}Q)^*(T\widetilde{E})$. In view of~\eqref{eq:block_decomposition_E}, the latter is the case if, and only if,~$T^{-*}Q$ allows a block-representation
\begin{align*}
T^{-*}Q = \begin{bmatrix}
Q^1\\
Q^2
\end{bmatrix}
\end{align*} 
for some~$Q^2\in\mathbb{K}^{(\ell-n)\times n}$ and some Hermitian~$Q^1\in\mathbb{K}^{n\times n}$.
Consider the skew-symmetric matrix
\begin{align*}
J_0 = T^{-1}\begin{bmatrix}
0 & -1\\
1 & 0 & -1\\
& \ddots & \ddots & \ddots\\
& & 1 & 0 & -1\\
& & & 1 & 0
\end{bmatrix}T^{-*}\in\mathbb{K}^{\ell\times \ell}
\end{align*}
and, for~$\beta\in(\mathbb{R}\setminus\set{0})^{n}$, the matrix
\begin{align*}
Q(\beta) := T^*\begin{bmatrix}
\beta_1\\
& \ddots\\
& & \beta_n\\
0 & \cdots & 0\\
\vdots & \ddots & \vdots\\
0 & \cdots & 0
\end{bmatrix}\in\mathbb{K}^{\ell\times n}.
\end{align*}
The matrices~$Q(\beta)$ fullfill, for all~$\beta\in\mathbb{R}^n$,~$\widetilde{E}^*Q(\beta) = Q(\beta)^*\widetilde{E}$.
Our choice of~$Q(\beta)$ and~$J_0$ guarantees that~$(J_0,Q(\beta))\in V_{(i)}(\widetilde{E})$ for all~$\beta\in\mathbb{R}^n$ and~$B\in\mathbb{K}^{\ell\times m}$. Furthermore,
\begin{align*}
J_0Q(\beta) = \begin{cases}
\begin{bmatrix}
0 & -\beta_2\\
\beta_1 & 0 & -\beta_3\\
& \ddots & \ddots & \ddots\\
& & \beta_{n-2} & 0 & -\beta_n\\
& & & \beta_{n-1} & 0
\end{bmatrix}, & \ell = n,\\
\left[\begin{array}{c}
\begin{array}{ccccc}
0 & -\beta_2\\
\beta_1 & 0 & -\beta_3\\
& \ddots & \ddots & \ddots\\
& & \beta_{n-2} & 0 & -\beta_n\\
& & & \beta_{n-1} & 0\\
& & & & \beta_n
\end{array}\\\hline
0_{(\ell-n-1)\times n}
\end{array}\right], & \ell>n.
\end{cases}
\end{align*}
Let~$\delta\in(\mathbb{R}\setminus\set{0})^{\ell-n}$ and~$\xi\in(\mathbb{R}\setminus\set{0})^{\ell-n+1}$, and put
\begin{align*}
B(\delta,\xi) := T^{-1}\left[\begin{array}{c|c}
0_{(n-1)\times (\ell-n+1)} & 0_{(n-1)\times (m-\ell+n-1)}\\\hline
\begin{array}{cccc}
\xi_1\\
\delta_1 & \xi_2\\
& \ddots & \ddots\\
& & \delta_{\ell-n} & \xi_{\ell-n+1}
\end{array} & 0_{(\ell-n)\times (m-\ell+n-1)}
\end{array}\right].
\end{align*}
Without loss of generality  we may assume that the~$\widetilde{M}_i$ are ordered in such a manner that
\begin{align*}
\widetilde{M}_1 = P(x)\mapsto \det [P(x)_{i,j}]_{i,j\in\mathbb{N}^*_{\leq \ell}},
\end{align*}
i.e.~$\sigma$ and~$\pi$ in Definition~\ref{def:minor} are the respective identity functions, and
\begin{align*}
\widetilde{M}_2 = P(x)\mapsto \det [P(x)_{i,j}]_{i,j-1\in\mathbb{N}^*_{\leq \ell}},
\end{align*}
i.e.~$\pi$ in Definition~\ref{def:minor} is the increment by one and~$\sigma$ the identity function. Then, our choice of~$J_0, Q(\beta)$ and~$B(\delta,\xi)$ yields
\begin{align*}
M_1^{T}(\widetilde{E},J_0,Q(\beta),B(\delta,\xi)) & = \widetilde{M}_1([xT\widetilde{E}-TJ_0T^{-*}TQ(\beta),B(\delta,\xi)])\\
& = x^n\prod_{j = 1}^{\ell-n}\delta_j  + \text{lower order terms}.
\end{align*}
Since~$\delta_j\neq 0$,~$j\in\mathbb{N}^*_{\leq\ell-n}$, we have~$\deg M_1^{T}(E,J_0,Q(\beta),B(\delta,\xi)) = n = \gamma_1^{T}$. Further, we have
\begin{align*}
M_2^{T}(\widetilde{E},J_0,Q(\beta),B(\delta,\xi)) = \prod_{i = 2}^n \beta_i\prod_{j = 1}^{\ell-n+1}\xi_j\neq 0.
\end{align*}
Thus, the equivalence~\eqref{eq:properties_p_polynomials} yields that~$p_{1,2}^{T}(\widetilde{E},J_0,Q(\beta),B(\delta,\xi))\neq 0$.
By definition of~$S_{(i)}(\widetilde{E})$ and~$\mathbb{V}$, we conclude~$(J_0,Q(\beta),B(\delta,\xi))\in S_{(i)}(\widetilde{E})$ and hence the latter set is non-empty. This proves the assertion for the case~$\ell<n+m$.
\\

\noindent
\textsc{Step}~(i$_5$):\quad 
Let~$\ell>n+m$. We show that~$S_{(i)}(\widetilde{E})$ is non-empty. 
\\
Analogously to Step~(i$_4$), 
we use the decomposition
\begin{align*}
T\widetilde{E} = \begin{bmatrix}
I_n\\
0_{(\ell-n)\times n}
\end{bmatrix}
\end{align*}
for some~$T\in\mathbf{Gl}(\mathbb{R}^\ell)$. Define, additionally to the matrices~$J_0$ and~$Q(\beta)$ defined in Step~(i$_4$), 
 the matrix
\color{black}
\begin{align*}
B_0' := T^{-1}\left[\begin{array}{c}
0_{n\times m}\\\hline
\begin{array}{cccc}
1\\
1 & 1\\
& \ddots & \ddots\\
& & 1 & 1\\
& & & 1
\end{array}\\\hline
0_{(\ell-n-m-1)\times m}
\end{array}\right].
\end{align*}
Analogously to the case~$\ell< n+m$ in 
 Step~(i$_4$),
  we can, without loss of generality,  assume that the minors~$\widetilde{M}_i$ are ordered in such a manner that
\begin{align*}
\widetilde{M}_1 = P(x)\mapsto\det[P(x)_{i,j}]_{i,j\in\mathbb{N}^*_{\leq n+m}}
\end{align*}
and
\begin{align*}
\widetilde{M}_2 = P(x)\mapsto\det[P(x)_{i,j}]_{i-1,j\in\mathbb{N}^*_{\leq n+m}}.
\end{align*}
Then, it is straightforward to verify
\begin{align*}
M_1^{T}(\widetilde{E},J_0',Q'(\mathbf{1}),B_0') = \det\underbrace{\begin{bmatrix}
x & 1 & & & & &\\
-1 & x & 1 & & &\\
& \ddots & \ddots & \ddots & & & \\
& & -1 & x & 1\\
& & & -1 & x
\end{bmatrix}}_{\in\mathbb{K}[x]^{n\times n}}.
\end{align*}
Since~$E$ has~$n$ columns, it is easy to see that the maximal degree of any minor of~$[xPE-PJQ,B]$ of order~$d$ does not exceed~$n$. Hence we have
\begin{align*}
\deg M_1^{P}(E,J_0,Q(\mathbf{1}),B_0') = n = \gamma_1^{T,S}.
\end{align*}
Further, we have
\begin{align*}
M_2^{P}(\widetilde{E},J_0,Q(\mathbf{1}),B_0') = (-1)^n\neq 0.
\end{align*}
This shows that~$p_{1,2}^{P}(\widetilde{E},J_0,Q(\mathbf{1}),B_0')\neq 0$ and therefore~$(J_0,Q(\mathbf{1}),B_0')\in S_{(i)}(\widetilde{E})$. This shows that the latter set is indeed non-empty. As discussed earlier, this shows that~$\mathbb{V}^c$ is indeed Euclidean dense in~$V_1'$.
\\

\noindent
(j)\quad 
The proof of (j) is analogous to the proof of (i): We can construct an algebraic set~$\widehat{\mathbb{V}}$ of the same type as in (i). It is just necessary to redefine the~$M_i^{P}$ as
\begin{align*}
M_i^{P} := (E,J,R,Q,B)\mapsto \widetilde{M}_i([xPET-P(J-R)P^*P^{-*}QT,B])
\end{align*}
to incorporate the~$R$. Then we show that for given~$E\in\mathbb{K}^{\ell\times n}$ with full rank the set of all matrices~$(J,R,Q,B)$ with~$(E,J,R,Q,B)\in\widehat{\mathbb{V}}^c$ has non-empty intersection with the convex set
\begin{align*}
\set{(J,R,Q,B)\,\big\vert\,(E,J,R,Q,B)\in\Sigma_{\ell,n,m}^{sdH}}.
\end{align*}
This can be directly concluded from the proof in
Step~(i$_3$)
since each matrix triple~$(J,Q,B)\in S(E)$ yields a feasible matrix quadruple~$(J,0,Q,B)$ for the proof of this step. This completes the proof of (a)--(j) for the case~$\ell\geq n$.
\\

\noindent
Case~$\ell < n$:  \ 
We prove (a)--(h).
\\
 In view of Lemma~\ref{lem:semidefinite_Hamiltonian_pair_pointwise_invertible}, the set
\begin{align*}
S^H := \set{(E,J,Q,B)\in \Sigma_{\ell,n,m}^H\,\big\vert\,\rk E = \ell}
\end{align*}
is relative generic in~$\Sigma_{\ell,n,m}^H$ and
\begin{align*}
S^{sdH} := \set{E,J,R,Q,B)\in\Sigma_{\ell,n,m}^{sdH}\,\big\vert\,\rk E = \ell}
\end{align*}
is relative generic in~$\Sigma_{\ell,n,m}^{sdH}$. Since~$\ell = \min\set{\ell,n+m} = \min\set{\ell,2n+m}$, the inclusions
\begin{align*}
S^H\subseteq S_{(a)}\cap S_{(b)}\cap S_{(c)}\cap S_{(d)}~\text{and}~S^{sdH}\subseteq S_{(e)}\cap S_{(f)}\cap S_{(g)}\cap S_{(h)}
\end{align*}
hold, where the inclusion~$S_{(a)}\subseteq S_{(c)}$ can be verified analogously to (c) in the case~$\ell\geq n$. Hence, Proposition~\ref{prop:properties_rel_gen}\,(b) yields that~$S_{(a)}$,~$S_{(b)}$,~$S_{(c)}$ and~$S_{(d)}$ are relative generic in~$\Sigma_{\ell,n,m}^H$, and~$S_{(e)}$,~$S_{(f)}$,~$S_{(g)}$ and~$S_{(h)}$ are relative generic in~$\Sigma_{\ell,n,m}^{sdH}$. 
\\
 
\noindent
(i)\quad    
Analogously to the case~$\ell\geq n$, we construct the polynomials~$p_{i,j}$ associated to all minors~$\widetilde{M}_1,\ldots,\widetilde{M}_q$,~$q\in\mathbb{N}^*$, of order~$\ell$ w.r.t.~$\mathbb{R}[x]^{\ell\times (n+m)}$ and their induced functions
\begin{align*}
M_i^{P,T}:\mathbb{K}^{\ell\times n}\times\mathbb{K}^{\ell\times \ell}\times\mathbb{K}^{\ell\times n}\times\mathbb{K}^{\ell\times m}\to\mathbb{K}[x],\quad (E,J,Q,B)\mapsto \widetilde{M}_i([xPET-PJP^*P^{-*}QT,B]).
\end{align*}
We omit the details of this construction. The ideal~$I$ generated by the polynomials~$p_{i,j}^{P,T}$ defined as in~\eqref{eq:Sylvester_operator} generates the algebraic set
\begin{align*}
\mathbb{V} := \set{(E,J,Q,B)\,\big\vert\forall p\in I: p(E,J,Q,B) = 0}
\end{align*}
which fullfills~$\mathbb{V}^c\cap \Sigma_{\ell,n,m}^H\subseteq S_{(i)}\cap \Sigma_{\ell,n,m}^H$. By Definition~\ref{def:rel_gen}, it remains to prove that~$\mathbb{V}^c\cap\Sigma_{\ell,n,m}^H$ is Euclidean dense in~$\Sigma_{\ell,n,m}^H$. Let~$(E,J,Q,B)\in\Sigma_{\ell,n,m}^H$. As discussed in detail in the proof of (i) in the case~$\ell\geq n$, it suffices to show that
\begin{align*}
S(E) := \set{(J,Q,B)\in\mathbb{K}^{\ell\times \ell}\times\mathbb{K}^{\ell\times n}\times\mathbb{K}^{\ell\times m}\,\big\vert\,(E,J,Q,B)\in\mathbb{V}^c}
\end{align*}
is relative generic in
\begin{align*}
V(E) := \set{(J,Q,B)\in\mathbb{K}^{\ell\times \ell}\times\mathbb{K}^{\ell\times n}\times\mathbb{K}^{\ell\times m}\,\big\vert\,J^* = -J, Q^*E = E^*Q\geq 0}.
\end{align*}
Furthermore, an application of Lemma~\ref{lem:semidefinite_Hamiltonian_pair_pointwise_invertible} analogously to Lemma~\ref{lem:Hamiltonian_pairs_full_rank_pointwise}\,(ii) in the case~$\ell\geq n$ yields that it suffices to consider the case~$\rk E = \ell$. Since~$\mathbb{V}$ is an algebraic set,~$S(E)$ is Zariski-open. Since~$V(E)$ is a non-empty convex set, it suffices therefore in view of Lemma~\ref{lem:convex_reference_sets} to prove that~$S(E)\cap V(E) \neq\emptyset$. There are~$P\in\mathbf{Gl}(\mathbb{K}^\ell)$ and~$T\in\mathbf{Gl}(\mathbb{K}^n)$ so that
\begin{align*}
PET = [I_\ell, 0_{\ell\times (n-\ell)}]
\end{align*}
Choose
\begin{align*}
Q_1 := P^*[I_\ell, 0_{\ell\times (n-\ell)}]T^{-1}.
\end{align*}
Since~$T$ is regular, we conclude from
\begin{align*}
T^*E^*QT = T^*E^*P^*P^{-*}QT = \begin{bmatrix}
I_\ell & 0_{\ell\times (n-\ell)}\\
0_{(n-\ell)\times \ell} & 0_{(n-\ell)\times (n-\ell)}
\end{bmatrix}
\end{align*}
that~$E^*Q_1$ is Hermitian and positive semi-definite.
Choose further the skew-symmetric matrix
\begin{align*}
J_1 = P^{-1}\begin{bmatrix}
0 & -1\\
1 & 0 & -1\\
& \ddots & \ddots & \ddots\\
& & 1 & 0 & -1\\
& & & 1 & 0
\end{bmatrix}P^{-*}\in\mathbb{K}^{\ell\times \ell}
\end{align*}
and with~$e_\ell = [0,\ldots,0,1]^\top$,
\begin{align*}
\widehat{B} := P^{-1}[e_\ell,0_{\ell\times (m-1)}]\in\mathbb{K}^{\ell\times m}.
\end{align*}
With these particular matrices, we get
\begin{align*}
P\big[xE-J_1Q_1,\widehat{B}]\begin{bmatrix}
T\\& I_m
\end{bmatrix} = \left[\begin{array}{c|c|c|c}
\begin{array}{ccccc}
x & 1\\
-1 & x & 1\\
& \ddots & \ddots & \ddots\\
& & -1 & x & 1\\
& & & -1 & x
\end{array} & \, 0_{\ell\times (n-\ell)}\, & \begin{array}{c}
0\\
0\\
\vdots\\
0\\
1
\end{array} &\, 0_{\ell\times(m-1)}
\end{array}\right].
\end{align*}
Without loss of generality we can assume that the~$\widetilde{M}_i$ are ordered in such a manner that
\begin{align*}
\widetilde{M}_1 = P(x)\mapsto\det[P(x)_{i,j}]_{i,j\in\mathbb{N}^*_{\leq \ell}}
\end{align*}
and
\begin{align*}
\widetilde{M}_2 = P(x)\mapsto\det[P(x)_{i,j}]_{i\in\mathbb{N}^*_{\leq \ell},j\in\set{2,\ldots,\ell,n+1}}.
\end{align*}
Then, it is easy to verify that~$M_1^{P,T}(E,J_1,Q_1,\widehat{B})$ and~$M_2^{P,T}(E,J_1,Q_1,\widehat{B})$ are coprime and
\begin{align*}
\deg M_1^{P,T}(E,J_1,Q_1,\widehat{B}) = \gamma_1^{P,T}.
\end{align*}
We can, analogously to the case~$\ell \geq n$, conclude from~\eqref{eq:properties_p_polynomials} that~$p_{1,2}^{P,T}(E,J_1,Q_1,\widehat{B})\neq 0$. Thus, we have~$(J_1,Q_1,\widehat{B})\in S(E)\cap V(E)$ and therefore the latter is non-empty. We conclude that~$S(E)$ is relative generic in~$V(E)$. Therefore, we conclude that~$\mathbb{V}^c$ is indeed Euclidean dense in~$\Sigma_{\ell,n,m}^H$.
\\
 
\noindent
(j)\quad 
As in the case~$\ell\geq n$, we can modify the proof of (i) to incorporate the additional matrix~$R$ and readily conclude that~$S_{(j)}$ is relative generic in~$\Sigma_{\ell,n,m}^{sdH}$. This completes the proof of the proposition.
\end{proof}

\section{Relative genericity of controllability}\label{Sec:Main_results-contr}
We are now in the position to derive necessary and sufficient conditions on the dimensions~$\ell,n,m$ 
of  linear finite-dimensional port-Hamiltonian descriptor systems described by~\eqref{eq:pH_DAE}
so that they are relative generically controllable 
(with respect to the five concepts defined in Proposition~\ref{Def-Prop:Contr})
in the reference sets~$\Sigma_{\ell,n,m}^{H}$, ~$\Sigma_{\ell,n,m}^{sdH}$
and~$\Sigma_{\ell,n,m}^{dH}$.

\begin{Theorem}\label{thm:main_corollary}
Consider, for $\ell,n,m\in\mathbb{N}^*$,
the  port-Hamiltonian descriptor systems described by~\eqref{eq:pH_DAE},
subsets of controllable systems~$S_{\text{controllable}}^H$
defined in~\eqref{eq:Hcontr}, 
and reference sets~$\Sigma_{\ell,n,m}^{H}$, ~$\Sigma_{\ell,n,m}^{sdH}$
and~$\Sigma_{\ell,n,m}^{dH}$
defined in~\eqref{eq:Sigma-sdH}--\eqref{eq:Sigma-H}.
\color{black}
Then the following equivalences hold:
\begin{center}
\begin{tabular}{clccl}
(i) &~$S_{\text{freely initializable}}^H$ & is rel. gen. in~$\Sigma_{\ell,n,m}^{H}$ &~$\iff$ &~$\ell\leq n+m$,\\[2ex]
(ii) &~$S_{\text{impulse controllable}}^H$ & is rel. gen. in~$\Sigma_{\ell,n,m}^{H}$ &~$\iff$ &~$\ell\leq n+m$,\\[2ex]
(iii) &~$S_{\text{behavioural controllable}}^H$ & is rel. gen. in~$\Sigma_{\ell,n,m}^{H}$ &~$\iff$ &~$\ell\neq n+m$,\\[2ex]
(iv) &~$S_{\text{completely controllable}}^H$ & is rel. gen. in~$\Sigma_{\ell,n,m}^{H}$ &~$\iff$ &~$\ell<n+m$,\\[2ex]
(v) &~$S_{\text{strongly controllable}}^H$ & is rel. gen. in~$\Sigma_{\ell,n,m}^{H}$ &~$\iff$ &~$\ell <n+m$.\\
\end{tabular}
\end{center}
Moreover, if~$S_{\text{\textit{controllable}}}^H$ is not relative generic in~$\Sigma_{\ell,n,m}^H$, then~$(S_{\text{\textit{controllable}}}^H)^c$ is relative generic in~$\Sigma_{\ell,n,m}^H$, where \textit{controllable} stands for either one of the studied controllability notions. 

The above remains true, if we consider~$S_{\text{\textit{controllable}}}^{sdH}$ with reference set~$\Sigma_{\ell,n,m}^{sdH}$ or~$S_{\text{\textit{controllable}}}^{dH}$ with reference set~$\Sigma_{\ell,n,m}^{dH}$.
\end{Theorem}
\begin{proof}
In this proof, we will use the notions introduced in Proposition~\ref{lem:main_result_1}.

\noindent
(i) \quad By Proposition~\ref{lem:main_result_1} and Proposition~\ref{prop:properties_rel_gen}\,(c), the set~$S_{(a)}\cap S_{(b)}$ is relative generic in~$\Sigma_{\ell,n,m}^H$. In view of the algebraic characterization of impulse controllability in Proposition~\ref{Def-Prop:Contr}, we have the inclusions
\begin{align*}
S_{(a)}\cap S_{(b)}\subseteq\begin{cases}
S_{\textit{freely~initializable}}^H, & \ell\leq n+m,\\
(S_{\textit{freely~initializable}}^H)^c, & \ell>n+m.
\end{cases} 
\end{align*}
By Proposition~\ref{prop:properties_rel_gen}\,(b), we conclude that~$S_{\textit{freely~initializable}}^H$ is relative generic in~$\Sigma_{\ell,n,m}^H$ if, and only if,~$\ell\leq n+m$ and that~$(S_{\textit{freely~initializable}}^H)^c$ is relative generic in~$\Sigma_{\ell,n,m}^H$ if, and only if,~$\ell>n+m$.
\\

\noindent
(ii)\quad By Proposition~\ref{lem:main_result_1} and Proposition~\ref{prop:properties_rel_gen}\,(c), the set~$S_{(b)}\cap S_{(d)}$ is relative generic in~$\Sigma_{\ell,n,m}^H$. Proposition~\ref{Def-Prop:Contr} yields the inclusions
\begin{align*}
S_{(b)}\cap S_{(d)}\subseteq\begin{cases}
S_{\textit{impulse~controllable}}^H, & \ell\leq n+m,\\
(S_{\textit{impulse~controllable}}^H)^c, & \ell>n+m.
\end{cases} 
\end{align*}
Thus, Proposition~\ref{prop:properties_rel_gen}\,(b) yields the assertion for~$S_{\textit{impulse~controllable}}^H$.
\\

\noindent
(iii)\quad In view of Proposition~\ref{lem:main_result_1} and Proposition~\ref{prop:properties_rel_gen}\,(c), the set~$S_{(c)}\cap S_{(i)}$ is relative generic in~$\Sigma_{\ell,n,m}^H$ if, and only if,~$\ell\neq n+m$; otherwise the set~$S_{(c)}\cap S_{(i)}^c$ is relative generic in~$\Sigma_{\ell,n,m}^H$. By Proposition~\ref{Def-Prop:Contr}, the inclusion
\begin{align*}
S_{(c)}\cap S_{(i)}\subseteq S_{\textit{behavioural~controllable}}^H
\end{align*}
holds true. Thus, Proposition~\ref{prop:properties_rel_gen}\,(b) yields that~$S_{\textit{behavioural~controllable}}^H$ is relative generic in~$\Sigma_{\ell,n,m}^H$ if, and only if,~$\ell\neq n+m$; otherwise its complement is relative generic in~$\Sigma_{\ell,n,m}^H$.
\\

\noindent
(iv)\quad In view of Proposition~\ref{Def-Prop:Contr}, the identity
\begin{align*}
S_{\textit{completely~controllable}}^H = S_{\textit{freely~initializable}}^H\cap S_{\textit{behavioural~controllable}}^H
\end{align*}
holds true. By Proposition~\ref{prop:properties_rel_gen}\,(c), the assertion follows from (i) and (iii) of the present proposition.
\\

\noindent
(v)\quad Proposition~\ref{Def-Prop:Contr} yields that
\begin{align*}
S_{\textit{strongly~controllable}}^H = S_{\textit{impulse~controllable}}^H\cap S_{\textit{behavioural~controllable}}^H.
\end{align*}
Hence, Proposition~\ref{prop:properties_rel_gen}\,(c), and (ii) and (iii) of the present proposition imply the proposed equivalence for relative genericity of~$S_{\textit{strongly~controllable}}$.

The assertion for the reference set~$\Sigma_{\ell,n,m}^{sdH}$ can be analogously proven by replacing~$S_{(a)}, S_{(b)}, S_{(c)}, S_{(d)}$ and~$S_{(i)}$ by~$S_{(e)}, S_{(f)}, S_{(g)}, S_{(h)}$ and~$S_{(j)}$, respectively.

Finally, in view of Lemma~\ref{lem:definite_is_generic} and Proposition~\ref{prop:properties_rel_gen}\,(d),~$\Sigma_{\ell,n,m}^{dH}$ is a relative generic subset of~$\Sigma_{\ell,n,m}^{sdH}$. Therefore, Proposition~\ref{prop:properties_rel_gen}\,(e) yields the assertion for the reference set~$\Sigma_{\ell,n,m}^{dH}$. This shows the proposition.
\color{black}
\end{proof}

We finally observe that generic controllability holds for
port-Hamiltonian descriptor systems if, and only if, 
relative generic controllability  holds for unconstrained differential-algebraic systems.
This is made precise in the following corollary.

\begin{Corollary} \label{Cor:contr-DAE-pH}
Let $r,\ell,n,m\in\mathbb{N}^*$. 
Then the port-Hamiltonian descriptor systems~\eqref{eq:pH_DAE} 
are generically controllable if, and only if, the  unconstrained  descriptor systems~\eqref{eq:the_DAE} 
are relative generically controllable in~$\Sigma_{\ell,n,m}^{H}$, 
where controllable stands for either of
the five controllability concepts defined in Proposition~\ref{Def-Prop:Contr}.
\end{Corollary}
\begin{proof}
The corollary is a direct consequence of the conjunction of 
\cite[Thm.~2.3]{IlchKirc21} and Theorem~\ref{thm:main_corollary}.
\end{proof}

\section{Relative genericity of stabilizability}\label{Sec:Main_results-stab}
In this section we study~-- similar to relative generic  controllability in Section~\ref{Sec:Main_results-contr}~--
relative generic stabilizability.
Recall that there are three different concepts of stabilizability~--
 see Proposition~\ref{Prop:Stab}.
We   derive necessary and sufficient conditions on the dimensions~$\ell,n,m$ 
of  linear finite-dimensional port-Hamiltonian descriptor systems described by~\eqref{eq:pH_DAE}
so that they are relative generically stabilizable 
in the reference sets~$\Sigma_{\ell,n,m}^{H}$, ~$\Sigma_{\ell,n,m}^{sdH}$, 
and~$\Sigma_{\ell,n,m}^{dH}$.

\begin{Theorem}\label{thm:stab}
Let~$\ell,n,m\in\mathbb{N}^*$. 
Then   the following equivalences hold:
\begin{center}
\begin{tabular}{clccl}
(i) &~$S_{\text{behavioural stabilizable}}^H$ & is rel. gen. in~$\Sigma_{\ell,n,m}^{H}$ &~$\iff$ &~$\ell\neq n+m$,\\[2ex]
(ii) &~$S_{\text{completely stabilizable}}^H$ & is rel. gen. in~$\Sigma_{\ell,n,m}^{H}$ &~$\iff$ &~$\ell<n+m$,\\[2ex]
(iii) &~$S_{\text{strongly stabilizable}}^H$ & is rel. gen. in~$\Sigma_{\ell,n,m}^{H}$ &~$\iff$ &~$\ell <n+m$.\\
\end{tabular}
\end{center}
The above remains true, if we consider~$S_{\text{\textit{stabilizable}}}^{sdH}$ with reference set~$\Sigma_{\ell,n,m}^{sdH}$. Moreover, if~$\ell>n+m$, then~$(S_{\text{\textit{completely stabilizable}}}^*)^c$ and~$(S_{\text{\textit{strongly stabilizable}}}^*)^c$ are relative generic in~$\Sigma_{\ell,n,m}^*$, where~$*$ stands for either~$H$ or~$dH$.
\end{Theorem}
\begin{proof}
We distinguish the two cases~$\ell\neq n+m$ and~$\ell = n+m$.

\noindent
Case~$\ell\neq n+m$:  \ 
In view of Proposition~\ref{Prop:Stab}, the inclusions
\begin{align*}
S_{(a)}\cap S_{(b)}\cap S_{(i)}\subseteq \begin{cases} S_{\text{completely stabilizable}}^H, & \ell<n+m,\\
(S_{\text{completely stabilizable}}^H)^c, & \ell>n+m,
\end{cases}
\end{align*}
and
\begin{align*}
S_{(b)}\cap S_{(d)}\cap S_{(i)}\subseteq \begin{cases} S_{\text{strongly stabilizable}}^H, & \ell<n+m,\\
(S_{\text{strongly stabilizable}}^H)^c, & \ell>n+m,
\end{cases}
\end{align*}
and
\begin{align*}
S_{(c)}\cap S_{(i)}\subseteq S_{\text{behavioural stabilizable}}^H
\end{align*}
hold true. Therefore, Proposition~\ref{lem:main_result_1} and Proposition~\ref{prop:properties_rel_gen}\,(b) and (c) yields the assertion for~$\ell\neq n+m$. 

\noindent
Case~$\ell = n+m$:  \ 
 We show that the set
\begin{align*}
S^H := \set{(E,J,Q,B)\in\Sigma_{\ell,n,m}^H\,\big\vert\,\exists\lambda\in\mathrm{int}\,\overline{\mathbb{C}}_+: \rk[\lambda E-JQ,B] < \ell = \min\set{\ell,n+m}}
\end{align*}
contains a inner point w.r.t.\ the Euclidean relative topology on~$\Sigma_{\ell,n,m}^H$. Then, we can show that the complement of~$S_{\text{\textit{controllable}}}^H$ has a relative inner point and is therefore especially not nowhere dense.

In passing we note that~$\ell = n+m\geq 2$. 
Since the spectrum of a matrix depends continuously from its entries,~$S^H$ is open and thus it suffices to show that~$S^H\neq\emptyset$. Put
\begin{align*}
E: = \begin{bmatrix}
I_n\\
0_{m\times n}
\end{bmatrix},
\ 
J := \begin{bmatrix}
0 & -1\\
1 & 0 & -1\\
& \ddots & \ddots & \ddots\\
& & 1 & 0 & -1\\
& & & 1 & 0
\end{bmatrix}\in\mathbb{K}^{\ell\times \ell},
\ 
Q := \begin{bmatrix}
0_{n\times n}\\
-e_n^\top\\
0_{(m-1)\times n}
\end{bmatrix},
\ 
B := \begin{bmatrix}
0_{n\times m}\\
I_m
\end{bmatrix}.
\end{align*}
Then~$E^*Q = 0_{n\times n}\in\mathbb{R}^{n\times n}$ is symmetric and positive semi-definite.
Furthermore,
\begin{align*}
JQ = \begin{cases}\begin{bmatrix}
0_{(n-1)\times n}\\
e_n^\top\\
0_{1\times n}\\
-e_n^\top\\
0_{(m-2)\times n}
\end{bmatrix}, & m\geq 2  \\[8ex]
\begin{bmatrix}
0_{(n-1)\times n}\\
e_n^\top\\
0_{1\times n}
\end{bmatrix}, & m = 1. 
\end{cases}
\end{align*}
This yields
\begin{align*}
[xE-JQ,B] = \begin{cases} 
\begin{bmatrix}
x\\
& &\ddots\\
& & &&x & \\
& & && & x-1\\
& & && & 0 &1\\
& & && & 1 & 0 & 1\\
& & & && & && \ddots\\
& & & & &&&&& 1
\end{bmatrix}, & m\geq 2\\[11ex]
\begin{bmatrix}
x\\
& &\ddots\\
& & &&x & \\
& & && & x-1\\
& & && & 0 &1
\end{bmatrix}, & m = 1.
\end{cases}
\end{align*}
Hence, we conclude
\begin{align*}
\det[xE-JQ,B] = x^{n-1}(x-1),
\end{align*}
which vanishes at~$x=1$; equivalently, since $[xE-JQ,B]\in\mathbb{K}[x]^{\ell\times \ell}$, $\rk[1\cdot E-JQ,B]<\ell$. Thus~$(E,J,Q,B)\in S^H$ and the latter set is non-empty. In view of Proposition~\ref{lem:main_result_1} and Proposition~\ref{prop:properties_rel_gen}\,(c), the sets~$S_{(a)}\cap S_{(b)}$,~$S_{(b)}\cap S_{(d)}$ and~$S_{(c)}$ are relative generic in~$\Sigma_{\ell,n,m}^H$. Therefore, we conclude that the sets~$S_{(a)}\cap S_{(b)}\cap S^H$,~$S_{(b)}\cap S_{(d)}\cap S^H$ and~$S_{(c)}\cap S^H$ have non-empty relative interior.
By Proposition~\ref{Def-Prop:Contr}, the inclusions
\begin{align*}
S_{(a)}\cap S_{(b)}\cap S^H & \subseteq (S_{\text{completely stabilizable}}^H)^c\\
S_{(b)}\cap S_{(d)}\cap S^H & \subseteq (S_{\text{strongly stabilizable}}^H)^c\\
S_{(c)}\cap S^H & \subseteq (S_{\text{behavioural stabilizable}}^H)^c
\end{align*}
hold true. Hence, the sets on the righthandside contain an inner point and are especially not nowhere dense. By Proposition~\ref{prop:properties_rel_gen}\,(a), we conclude that neither of the sets~$S_{\text{completely stabilizable}}^H$, $S_{\text{strongly stabilizable}}^H$ and~$S_{\text{behavioural stabilizable}}^H$ are relative generic in~$\Sigma_{\ell,n,m}^H$.

The respective statement for~$\Sigma_{\ell,n,m}^{sdH}$ can be proven similiarly by considering
\begin{align*}
S^{sdH} := \set{(E,J,Q,B)\in\Sigma_{\ell,n,m}^{sdH}\,\big\vert\,\exists\lambda\in\overline{\mathbb{C}}_+: \rk[\lambda E-JQ,B] < \min\set{\ell,n+m}}.
\end{align*}
Since~$\Sigma_{\ell,n,m}^{dH}$ is a relative generic subset of~$\Sigma_{\ell,n,m}^{sdH}$, Proposition~\ref{prop:properties_rel_gen}\,(e) yields the assertion for the reference set~$\Sigma_{\ell,n,m}^{dH}$. This completes the proof of the theorem.
\color{black}
\end{proof}

Finally we observe~-- as for controllability in Corollary~\ref{Cor:contr-DAE-pH}~--
that port-Hamiltonian descriptor systems~\eqref{eq:pH_DAE}  are generically stabilizable
if, and only if, 
the  
 unconstrained systems~\eqref{eq:the_DAE} 
are relative generically stabilizable in~$\Sigma_{\ell,n,m}^{H}$.

\begin{Corollary} \label{Cor:stab-DAE-pH}
Let~$\ell,n,m\in\mathbb{N}^*$. 
Then the port-Hamiltonian descriptor systems~\eqref{eq:pH_DAE} 
are generically stablizable  if, and only if, the   unconstrained  descriptor systems~\eqref{eq:the_DAE}
are relative generically stablizable in~$\Sigma_{\ell,n,m}^{H}$, 
where stablizable stands for either  of the three concepts in 
Proposition~\ref{Prop:Stab}.
\end{Corollary}

\begin{proof}
The corollary is a direct consequence of the conjunction of 
\cite[Thm.~3.3]{IlchKirc21} and Theorem~\ref{thm:stab}.
\end{proof}

\section{Conclusion and outlook}\label{sec:outlook}
We have studied relative genericity of controllability and stabilizability for port-Hamiltonian systems and derived necessary and sufficient conditions for relative genericity of five controllability and three stabilizability concepts. This extends the result of~\cite[Theorem II.1]{Kirc21pp} to descriptor systems with arbitrary dimensions and nontrivial dissipation matrix. 

For square port-Hamiltonian systems, that is,~$\ell = n$, the matrix~$E$ in~\eqref{eq:pH_DAE} is generically invertible. In this case, Theorem~\ref{thm:main_corollary} is very closely related to~\cite[Theorem II.1]{Kirc21pp}, although with the addition of a possibly (generically) nonzero~$R$. In future work, we will put particular care on the rank of~$E$ and consider the reference sets
\begin{align*}
\Sigma_{\ell,n,m}^{H,\leq r} := \set{(E,J,Q,B)\in\Sigma_{\ell,n,m}^{H}\,\big\vert\,\rk E\leq r}
\end{align*}
and
\begin{align*}
\Sigma_{\ell,n,m}^{dH,\leq r} := \set{(E,J,R,Q,B)\in\Sigma_{\ell,n,m}^{dH}\,\big\vert\,\rk E\leq r}.
\end{align*}
When~$\ell = n$ and~$r<n$, this excludes especially the case that the considered systems can be (generically) reduced to an ODE system. We are confident that we can extend~\cite[Theorems~3.2 and~4.2]{IlchKirc22}, where unstructured DAEs with the same restriction on the rank of~$E$ are considered, to these reference sets for port-Hamiltonian descriptor systems.

\color{black}
\vspace*{1cm}
\begin{footnotesize}
\textbf{Acknowledgements} We are indebted to our colleague Karl Worthmann~(TU~Ilmenau) for several constructive discussions.
\end{footnotesize}

\newpage

\bibliographystyle{spmpsci}   
\bibliography{MST,Ergaenzungen-2022-08-30}
\end{document}